\documentclass[11pt]{article}

\usepackage{epsfig}
\usepackage{latexsym}
\usepackage{url}
\usepackage{float}
\usepackage{texnansi}
\usepackage{color}
\usepackage{tikz}
\usepackage{subfigure}
\usepackage{afterpage}
\usepackage{enumerate}
\usepackage[boxed]{algorithm}
\usepackage{algpseudocode}
\usepackage[normalem]{ulem}
\usepackage{lmodern}
\usepackage{booktabs}
\usepackage{sectsty}
\usepackage{ifthen}
\usepackage[leqno]{amsmath}
\usepackage{amsthm}
\usepackage{amssymb}
\usepackage{amsfonts}
\usepackage[margin=10pt,font=small,labelfont=bf]{caption}
\usepackage{natbib}
\newcommand{\field}[1]{\ensuremath{\mathbb{#1}}}
\newcommand{\R}{\ensuremath{\field{R}}} 
\newcommand{\1}{\ensuremath{\mathbf{1}}} 
\newcommand{\I}[1]{\ensuremath{\mathbb{I}_{\left\{#1\right\}}}} 
\newcommand{\tends}{\ensuremath{\rightarrow}} 
\newcommand{\PR}{\ensuremath{\mathsf{P}}} 
\newcommand{\E}{\ensuremath{\mathsf{E}}} 
\newcommand{\defeq}{\ensuremath{\triangleq}}
\newcommand{\subjectto}{\text{\rm subject to}} 
\newcommand{\Ascr}{\ensuremath{\mathcal A}}

\newcommand{\Fscr}{\ensuremath{\mathcal F}}
\newcommand{\Gscr}{\ensuremath{\mathcal G}}

\newcommand{\Nscr}{\ensuremath{\mathcal N}}

\newcommand{\Sscr}{\ensuremath{\mathcal S}}

\newcommand{\Xscr}{\ensuremath{\mathcal X}}

\DeclareMathOperator*{\argmin}{\mathrm{argmin}}
\DeclareMathOperator*{\argmax}{\mathrm{argmax}}
\newcommand{\minimize}{\ensuremath{\mathop{\mathrm{minimize}}\limits}}
\newcommand{\maximize}{\ensuremath{\mathop{\mathrm{maximize}}\limits}}
\newtheoremstyle{thm-sf}{}{}{\itshape}{}{\sffamily\bfseries}{.}{ }{}
\theoremstyle{thm-sf}

\newtheorem{assumption}{Assumption}

\newtheorem{theorem}{Theorem}

\newtheorem{corollary}{Corollary}
\newtheorem{lemma}{Lemma}

\usetikzlibrary{arrows,patterns,plotmarks}
\tikzstyle{every picture} += [>=stealth]
\allsectionsfont{\sffamily}
\makeatletter
\def\@seccntformat#1{\csname the#1\endcsname.\quad}
\makeatother

\floatname{algorithm}{\normalfont\sffamily\bfseries Algorithm}
\floatstyle{ruled}
\newcommand{\emailhref}[1]{\href{mailto:#1}{\tt #1}} 
\provideboolean{fastcompile}
\newcommand{\hidefastcompile}[1]{\ifthenelse{\boolean{fastcompile}}{}{#1}}

\usepackage{setspace}
\setcitestyle{authoryear,round,semicolon,aysep={,},yysep={,},notesep={, }}
\usepackage{pgfplots}
\usetikzlibrary{calc}

\newtheorem*{le:ell}{Lemma~\ref{le:ell}}
\newtheorem*{le:feas}{Lemma~\ref{le:feas}}
\newtheorem*{le:sample_complexity}{Lemma~\ref{le:sample_complexity}}

\title{\textsf{\textbf{Approximate Dynamic Programming \\
      via a Smoothed Linear Program}}}

\author{
  {\sffamily Vijay V. Desai} \\
  Industrial Engineering and Operations Research \\
  Columbia University \\
  email: \emailhref{vvd2101@columbia.edu} \\
  \and
  {\sffamily Vivek F. Farias} \\
  Sloan School of Management \\
  Massachusetts Institute of Technology \\
  email: \emailhref{vivekf@mit.edu}\\
  \and
  {\sffamily Ciamac C. Moallemi} \\
  Graduate School of Business \\
  Columbia University \\
  email: \emailhref{ciamac@gsb.columbia.edu}}
\date{Initial Version: August 4, 2009 \\
  Current Revision: September 25, 2009}

\begin{document}
\maketitle

\singlespacing

\begin{abstract}
  We present a novel linear program for the approximation of the dynamic
  programming cost-to-go function in high-dimensional stochastic control
  problems. LP approaches to approximate DP have typically relied on a natural
  `projection' of a well studied linear program for exact dynamic
  programming. Such programs restrict attention to approximations that are
  lower bounds to the optimal cost-to-go function. Our program---the `smoothed
  approximate linear program'---is distinct from such approaches and relaxes
  the restriction to lower bounding approximations in an appropriate fashion
  while remaining computationally tractable. Doing so appears to have several
  advantages: First, we demonstrate substantially superior bounds on the
  quality of approximation to the optimal cost-to-go function afforded by our
  approach. Second, experiments with our approach on a challenging problem
  (the game of Tetris) show that the approach outperforms the existing LP
  approach (which has previously been shown to be competitive with several ADP
  algorithms) by an order of magnitude.
\end{abstract}

\onehalfspacing

\section{Introduction}

Many dynamic optimization problems can be cast as Markov decision problems
(MDPs) and solved, in principle, via dynamic programming. Unfortunately, this
approach is frequently untenable due to the `curse of
dimensionality'. Approximate dynamic programming (ADP) is an approach which
attempts to address this difficulty.  ADP algorithms seek to compute good
approximations to the dynamic programing optimal cost-to-go function within
the span of some pre-specified set of basis functions.

ADP algorithms are typically motivated by exact algorithms for dynamic
programming. The approximate linear programming (ALP) method is one
such approach, motivated by the LP used for the computation of the optimal
cost-to-go function. Introduced by \citet{Schweitzer} and analyzed and further developed by
\citet{deFarias1,deFarias2}, this approach is attractive for a number of
reasons. First, the availability of efficient solvers for linear programming
makes the LP approach easy to implement. Second, the approach offers
attractive theoretical guarantees. In particular, the quality of the
approximation to the cost-to-go function produced by the LP approach can be
shown to compete, in an appropriate sense, with the quality of the best
possible approximation afforded by the set of basis functions used. A
testament to the success of the LP approach is the number of applications it
has seen in recent years in large scale dynamic optimization problems. These
applications range from the control of queueing networks to revenue management
to the solution of large scale stochastic games.

The optimization program employed in the ALP approach is in some sense the most natural linear programming formulation for ADP. In particular, the ALP is identical to the
linear program used for exact computation of the optimal cost-to-go function,
with further constraints limiting solutions to the low-dimensional subspace
spanned by the basis functions used. The resulting LP implicitly restricts
attention to approximations that are lower bounds to the optimal cost-to-go
function. The structure of this program appears crucial in establishing
guarantees on the quality of approximations produced by the approach; these approximation guarantees were remarkable and a first for any ADP method. 
That said, the restriction to lower bounds naturally leads one to ask whether the program
employed by the ALP approach is the `right' math programming formulation for ADP. In particular, it may be
advantageous to relax the lower bound requirement so as to allow for a better
approximation, and, ultimately, better policy performance. Is there an alternative formulation that permits better
approximations to the cost-to-go function while remaining computationally
tractable? Motivated by this question, the present paper introduces a new linear
program for ADP we call the `smoothed' approximate linear program (or
SALP). We believe that the SALP provides a preferable math programming formulation for ADP. In particular, we make the following
contributions:

\begin{enumerate}
\item  We are able to establish strong approximation and
  performance guarantees for approximations to the cost-to-go function produced by the SALP; these guarantees are
  \emph{substantially} stronger than the corresponding guarantees for the ALP.

\item The number of constraints and variables in the SALP scale with the size
  of the MDP state space. We nonetheless establish sample complexity bounds
  that demonstrate that an appropriate `sampled' SALP provides a good
  approximation to the SALP solution with a tractable number of sampled MDP
  states. Moreover, we identify structural properties for the sampled SALP
  that can be exploited for fast optimization. Our sample complexity results
  and these structural observations allow us to conclude that the SALP is
  essentially no harder to solve than existing LP formulations for ADP.

\item We present a computational study demonstrating the efficacy of our
  approach on the game of Tetris. Tetris is a notoriously difficult,
  `unstructured' dynamic optimization problem and has been used as a
  convenient testbed problem for numerous ADP approaches. The ALP has been
  demonstrated to be competitive with other ADP approaches for Tetris, such as
  temporal difference learning or policy gradient methods
  \citep[see][]{FariasVanRoyTetris}.  In detailed comparisons with the ALP, we
  show that the SALP provides an \emph{order of magnitude} improvement
  over controllers designed via that approach for the game of Tetris.
\end{enumerate}


The literature on ADP algorithms is vast and we make no attempt to survey it
here. \citet{BVRNDPReview} or \citet[][Chap.~6]{Bertsekas95} provide good,
brief overviews, while \citet{BertsekasNDP} and \citet{Powell2007} are
encyclopedic references on the topic.  The exact LP for the solution of
dynamic programs is attributed to \citet{Manne}. The ALP approach to ADP was
introduced by \citet{Schweitzer} and
\citet{deFarias1,deFarias2}. \citet{deFarias1} establish strong approximation
guarantees for ALP based approximations assuming knowledge of a
`Lyapunov'-like function which must be included in the basis. The approach we
present may be viewed as optimizing over all possible Lyapunov functions.
\citet{deFarias06} introduce a program for average cost approximate dynamic
programming that resembles the SALP; a critical difference is that their
program requires the relative violation allowed across ALP constraints be
specified as input. Applications of the LP approach to ADP range from
scheduling in queueing networks \citep{Morrison99,Veatch05,Moallemi08},
revenue management \citep{Adelman05,FariasVanRoyNetRM, Zhang09}, portfolio
management \citep{HanThesis}, inventory problems \citep{Adelman04,Adelman09},
and algorithms for solving stochastic games \citep{Weintraub08} among others.
Remarkably, in applications such as network revenue management, control
policies produced via the LP approach
\citep[namely,][]{Adelman05,FariasVanRoyNetRM} are competitive with ADP
approaches that carefully exploit problem structure, such as for instance
that of \citet{Topaloglu07}.



The remainder of this paper is organized as follows: In
Section~\ref{sec:prob-form}, we formulate the approximate dynamic programming
setting and describe the ALP approach. The smoothed ALP is developed as a
relaxation of the ALP in Section~\ref{sec:salp}. Section~\ref{sec:analysis}
provides a theoretical analysis of the SALP, in terms of approximation and
performance guarantees, as well as a sample complexity bound. In
Section~\ref{sec:prac-imp}, we describe the practical implementation of the
SALP method, illustrating how parameter choices can be made as well as how to
efficiently solve the resulting optimization program. Section~\ref{sec:tetris}
contains the computational study of the game Tetris. Finally, in
Section~\ref{sec:conclusion}, we conclude.

\section{Problem Formulation}\label{sec:prob-form}

Our setting is that of a discrete-time, discounted infinite-horizon,
cost-minimizing MDP with a finite state space $\Xscr$ and finite action space
$\Ascr$. At time $t$, given the current state $x_t$ and a choice of action
$a_t$, a per-stage cost $g(x_t,a_t)$ is incurred. The subsequent state
$x_{t+1}$ is determined according to the transition probability kernel
$P_{a_t}(x_t,\cdot)$. 

A stationary policy $\mu\colon \Xscr\rightarrow\Ascr$ is a mapping that determines
the choice of action at each time as a function of the state. Given each
initial state $x_0=x$, the expected discounted cost (cost-to-go function) of
the policy $\mu$ is given by
\[
J_\mu(x) \defeq
\E_\mu\left[\left.
\sum_{t=0}^\infty \alpha^t g(x_t,\mu(x_t))\ \right|\ x_0=x\right].
\]
Here, $\alpha\in(0,1)$ is the discount factor. The expectation is taken under
the assumption that actions are selected according to the policy $\mu$. In
other words, at each time $t$, $a_t \defeq \mu(x_t)$.

Denote by $P_\mu\in\R^{\Xscr\times\Xscr}$ the transition probability matrix
for the policy $\mu$, whose $(x,x')$th entry is $P_{\mu(x)}(x,x')$. Denote by
$g_\mu\in\R^\Xscr$ the vector whose $x$th entry is $g(x,\mu(x))$. Then, the
cost-to-go function $J_\mu$ can be written in vector form as
\[
J_\mu = \sum_{t=0}^\infty \alpha^t P_\mu^t g_\mu.
\]
Further, the cost-to-go function $J_\mu$ is the unique solution to the
equation $T_\mu J = J$, where the operator $T_\mu$ is defined by $T_\mu J =
g_\mu + \alpha P_\mu J$.

Our goal is to find an optimal stationary policy $\mu^*$, that is, a policy
that minimizes the expected discounted cost from every state $x$. In particular,
\[
\mu^* \in \argmin_\mu J_\mu(x).
\]
The Bellman operator $T$ is defined component-wise according to 
\[
(TJ)(x) \defeq \min_{a\in\Ascr}\ g(x,a) + \alpha \sum_{x'\in\Xscr} P_a(x,x') J(x'),
\quad\forall\ x\in\Xscr.
\]
Bellman's equation is then the fixed point equation
\begin{equation}\label{eq:bellman}
TJ = J.
\end{equation}
Standard results in dynamic programming establish that the optimal cost-to-go
function $J^*$ is the unique solution to Bellman's equation \citep[see, for
example,][Chap.~1]{Bertsekas95}. Further, if $\mu^*$ is a policy that is
greedy with respect to $J^*$ (i.e., $\mu^*$ satisfies $T J^*=T_{\mu^*} J^*$),
then $\mu^*$ is an optimal policy.

\subsection{The Linear Programming Approach}\label{sec:exactlp}

A number of computational approaches are available for the solution of the
Bellman equation. One approach involves solving the optimization program:
\begin{equation}\label{eq:exactlp}
\begin{array}{lll}
\maximize_J & \nu^\top J \\
\subjectto & J \leq TJ.
\end{array}
\end{equation}
Here, $\nu\in\R^\Xscr$ is a vector with positive components that are known as
the {\em state-relevance weights}. The above program is indeed an LP since for each state $x$, the constraint $J(x) \leq (TJ)(x)$ is
equivalent to the set of $|\Ascr|$ linear constraints
\[
J(x) \leq g(x,a) + \alpha \sum_{x'\in\Xscr} P_a(x,x') J(x'),\quad\forall\
a\in\Ascr.
\]
We refer to \eqref{eq:exactlp} as the {\em exact LP}.

Suppose that a vector $J$ is feasible for the exact LP
\eqref{eq:exactlp}. Since $J \leq TJ$, monotonicity of the Bellman
operator implies that $J \leq T^k J$, for any integer $k \geq 1$. Since the
Bellman operator $T$ is a contraction, $T^k J$ must converge to the unique
fixed point $J^*$ as $k\tends\infty$. Thus, we have that $J \leq J^*$. Then, it
is clear that every feasible point for \eqref{eq:exactlp} is a component-wise
lower bound to $J^*$. Since $J^*$ itself is feasible for \eqref{eq:exactlp},
it must be that $J^*$ is the unique optimal solution to the exact LP.

\subsection{The Approximate Linear Program}\label{sec:alp}

In many problems, the size of the state space is enormous due to the curse of
dimensionality. In such cases, it may be prohibitive to store, much less
compute, the optimal cost-to-go function $J^*$. In approximate dynamic
programming (ADP), the goal is to find tractable approximations to the optimal
cost-to-go function $J^*$, with the hope that they will lead to good policies.

Specifically, consider a collection of {\em basis functions} $\{
\phi_1,\ldots,\phi_K\}$ where each $\phi_i\colon \Xscr\rightarrow \R$ is a
real-valued function on the state space. ADP algorithms seek to find linear
combinations of the basis functions that provide good approximations to the
optimal cost-to-go function. In particular, we seek a vector of weights
$r\in\R^K$ so that
\[
J^*(x) \approx J_r(x) \defeq \sum_{r=1}^K \phi_i(x) r_i = \Phi r(x).
\]
Here, we define $\Phi \defeq [\phi_1 \ \phi_2 \ \dots \ \phi_K]$ to be a matrix with
columns consisting of the basis functions. Given a vector of weights $r$ and
the corresponding value function approximation $\Phi r$, a policy $\mu_r$ is
naturally defined as the `greedy' policy with respect to $\Phi r$, i.e. as $T_{\mu_r} \Phi r = T \Phi r$.

One way to obtain a set of weights is to solve the exact LP
\eqref{eq:exactlp}, but restricting to the low-dimensional subspace of vectors
spanned by the basis functions. This leads to the {\em approximate linear
  program} (ALP), which is defined by
\begin{equation}\label{eq:alp}
\begin{array}{lll}
\maximize_r & \nu^\top \Phi r \\
\subjectto & \Phi r \leq T\Phi r.
\end{array}
\end{equation}

For the balance of the paper, we will make the following assumption:
\begin{assumption}
  Assume the $\nu$ is a probability distribution ($\nu \geq 0$, $\1^\top \nu =
  1$), and that the constant function $\1$ is in the span of the basis
  functions $\Phi$.
\end{assumption}

The geometric intuition behind the ALP is illustrated in
Figure~\ref{fig:alp}. Supposed that $r_{\text{ALP}}$ is a vector that is
optimal for the ALP. Then the approximate value function $\Phi r_{\text{ALP}}$
will lie on the subspace spanned by the columns of $\Phi$, as illustrated by
the orange line. $\Phi r_{\text{ALP}}$ will also satisfy the constraints of
the exact LP, illustrated by the dark gray region. By the discussion in Section~\ref{sec:exactlp}, this implies that $\Phi
r_{\text{ALP}} \leq J^*$. In other words, the approximate cost-to-go function
is necessarily a point-wise lower bound to the true cost-to-go function in the
span of $\Phi$.

One can thus interpret the ALP solution $r_{\text{ALP}}$ equivalently as the
optimal solution to the program
\begin{equation}\label{eq:state_rel_wts}
    \begin{array}{lll}
    \minimize_r & \| J^* - \Phi r \|_{1,\nu} \\
    \subjectto & \Phi r \leq T\Phi r.
    \end{array}
\end{equation}
Here, the weighted $1$-norm in the objective is defined by
\[
\| J^* - \Phi r \|_{1,\nu} \defeq \sum_{x\in\Xscr} \nu(x) | J^*(x) - \Phi r (x)|.
\]
This implies that the approximate LP will find the closest approximation (in
the appropriate norm) to the optimal cost-to-go function, out of all
approximations satisfying the constraints of the exact LP.

\begin{figure}[htb]
  \centering
  \subfigure[ALP case.\label{fig:alp}]{
  \hidefastcompile{
    \centering
  \begin{tikzpicture}[x=2.25in,y=2in]
    \node[coordinate] at (0,0) (origin) {};
    \node[coordinate] at (1,0) (x1) {};
    \node[coordinate] at ([xshift=-15pt] origin) (x1m) {};
    \node[coordinate] at (x1m -| x1) (x1e) {};
    \node[coordinate] at (0,1) (x2) {};
    \node[coordinate] at ([yshift=-15pt] origin) (x2m) {};
    \node[coordinate] at (x2m |- x2) (x2e) {};
    \node[coordinate] at (x1m |- x2m) (origin_m) {};

    \node[coordinate] at (0.5,0.75) (Jstar) {};
    \node[coordinate] at (x1m |- 0,0.05) (bound_l) {};
    \node[coordinate] at (x2m -| 0.45,0) (bound_d) {};

    \node[coordinate] at ([xshift=25pt,yshift=25pt] Jstar) (Jstar2) {};
    \node[coordinate] at (x1m |- 0,0.9) (bound_l2) {};
    \node[coordinate] at ([xshift=25pt] bound_d) (bound_d2) {};

    \fill[gray!50] (bound_l) -- (Jstar) -- (bound_d)
    -- (origin_m) -- cycle;
    \draw[thick] (bound_l) -- (Jstar) -- (bound_d);

    \node[coordinate] at (x2 -| 0.4,1) (phir) {};
    \node[above] at (phir) { $J = \Phi r$};
    \node[coordinate] at (intersection of origin_m--x2m and phir--origin)
    (phir_b) {};
    \draw[thick,orange] (phir_b) -- (phir);
    \node[coordinate] at (intersection of phir_b--phir and bound_l--Jstar)
    (phir_alp) {};
    \draw[fill=red] (phir_alp) circle (2pt);
    \node[below right] at (phir_alp) {$\Phi r_{\text{ALP}}$};
    \node[coordinate] at (intersection of phir_b--phir and bound_l2--Jstar2)
    (phir_salp) {};

    \draw[fill=green] (Jstar) circle (2pt);
    \node[right] at (Jstar) {$J^*$};

    \draw[<->,red,very thick,shorten >=3pt,shorten <=3pt] (Jstar) to (phir_alp);

    \node[coordinate] at (0.3,0.1) (c) {};
    \draw[->,thick] (origin) to (c);
    \node[right] at (c) {$\nu$};

    \draw[->,thin] ([xshift=-5pt] origin) to (x1);
    \node[right] at (x1) {\footnotesize $J(1)$};
    \draw[->,thin] ([yshift=-5pt] origin) to (x2);
    \node[above] at (x2) {\footnotesize $J(2)$};
  \end{tikzpicture}
  }
  }
  \subfigure[SALP case.\label{fig:salp}]{
  \hidefastcompile{
    \centering
  \begin{tikzpicture}[x=2.25in,y=2in]
    \node[coordinate] at (0,0) (origin) {};
    \node[coordinate] at (1,0) (x1) {};
    \node[coordinate] at ([xshift=-15pt] origin) (x1m) {};
    \node[coordinate] at (x1m -| x1) (x1e) {};
    \node[coordinate] at (0,1) (x2) {};
    \node[coordinate] at ([yshift=-15pt] origin) (x2m) {};
    \node[coordinate] at (x2m |- x2) (x2e) {};
    \node[coordinate] at (x1m |- x2m) (origin_m) {};

    \node[coordinate] at (0.5,0.75) (Jstar) {};
    \node[coordinate] at (x1m |- 0,0.05) (bound_l) {};
    \node[coordinate] at (x2m -| 0.45,0) (bound_d) {};

    \node[coordinate] at ([xshift=25pt,yshift=25pt] Jstar) (Jstar2) {};
    \node[coordinate] at (x1m |- 0,0.9) (bound_l2) {};
    \node[coordinate] at ([xshift=25pt] bound_d) (bound_d2) {};

    \fill[gray!20] (bound_l2) -- (Jstar2) -- (bound_d2)
    -- (origin_m) -- cycle;
    \draw[thick] (bound_l2) -- (Jstar2) -- (bound_d2);

    \node[coordinate] at (x2 -| 0.4,1) (phir) {};
    \node[above] at (phir) { $J = \Phi r$};
    \node[coordinate] at (intersection of origin_m--x2m and phir--origin)
    (phir_b) {};
    \draw[thick,orange] (phir_b) -- (phir);
    \node[coordinate] at (intersection of phir_b--phir and bound_l2--Jstar2)
    (phir_salp) {};
    \draw[fill=blue] (phir_salp) circle (2pt);
    \node[below left] at (phir_salp) {$\Phi r_{\text{SALP}}$};

    \draw[fill=green] (Jstar) circle (2pt);
    \node[right] at (Jstar) {$J^*$};

    \draw[<->,red,very thick,shorten >=3pt,shorten <=3pt] (Jstar) to (phir_salp);

    \node[coordinate] at (0.3,0.1) (c) {};
    \draw[->,thick] (origin) to (c);
    \node[right] at (c) {$\nu$};

    \draw[->,thin] ([xshift=-5pt] origin) to (x1);
    \node[right] at (x1) {\footnotesize $J(1)$};
    \draw[->,thin] ([yshift=-5pt] origin) to (x2);
    \node[above] at (x2) {\footnotesize $J(2)$};
  \end{tikzpicture}
  }
  }
  \caption{A cartoon illustrating the feasible set and optimal solution for
    the ALP and SALP, in the case of a two-state MDP. The axes correspond to
    the components of the value function. A careful relaxation from the
    feasible set of the ALP to that of the SALP can yield an improved
    approximation. It is easy to construct a concrete two state example with the above features.} \label{fig:alp-salp}
\end{figure}
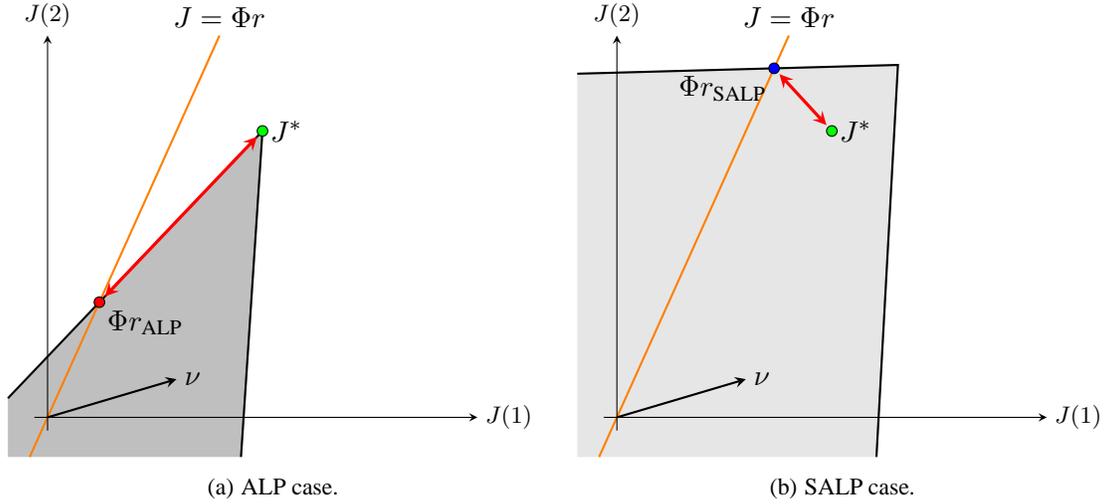

\section{The Smoothed ALP}\label{sec:salp}

The $J \leq TJ$ constraints in the exact LP, which carry over to the ALP,
impose a strong restriction on the cost-to-go function approximation: in
particular they restrict us to approximations that are lower bounds to $J^*$
at {\em every point in the state space}. In the case where the state space is
very large, and the number of basis functions is (relatively) small, it may be
the case that constraints arising from rarely visited or pathological states
are binding and influence the optimal solution.

In many cases, the ultimate goal is not to find a \emph{lower bound} on the
optimal cost-to-go function, but rather to find a \emph{good
  approximation}. In these instances, it may be that relaxing the
constraints in the ALP, so as not to require a uniform lower bound, may allow
for better overall approximations to the optimal cost-to-go function. This is
also illustrated in Figure~\ref{fig:alp-salp}. Relaxing the feasible region of
the ALP in Figure~\ref{fig:alp} to the light gray region in
Figure~\ref{fig:salp} would yield the point $\Phi r_{\text{SALP}}$ as an
optimal solution. The relaxation in this case is clearly beneficial; it allows
us to compute a better approximation to $J^*$ than the point $\Phi
r_{\text{SALP}}$.  

Can we construct a fruitful relaxation of this sort in
general?  The {\em smoothed approximate linear program} (SALP) is given by:
\begin{equation}\label{eq:salp}
\begin{array}{lll}
\maximize_{r, s} & \nu^\top \Phi r \\
\subjectto & \Phi r \leq T\Phi r + s, \\
& \pi^\top s \leq \theta,\quad s \geq 0.
\end{array}
\end{equation}
Here, a vector $s\in\R^\Xscr$ of additional decision variables has been
introduced. For each state $x$, $s(x)$ is a non-negative decision variable (a
slack) that allows for violation of the corresponding ALP constraint. The
parameter $\theta \geq 0$ is a non-negative scalar. The parameter
$\pi\in\R^\Xscr$ is a probability distribution known as the {\em constraint
  violation distribution}. The parameter $\theta$ is thus a {\em violation
  budget}: the expected violation of the $\Phi r \leq T \Phi r$ constraint,
under the distribution $\pi$, must be less than $\theta$.

The SALP can be alternatively written as
\begin{equation}\label{eq:salp2}
\begin{array}{lll}
\maximize_{r} & \nu^\top \Phi r \\
\subjectto & \pi^\top (\Phi r - T\Phi r)^+ \leq \theta.
\end{array}
\end{equation}
Here, given a vector $J$, $J^+(x) \defeq \max(J(x), 0)$ is defined
to be the component-wise positive part. Note that, when $\theta=0$, the SALP
is equivalent to the ALP. When $\theta > 0$, the SALP replaces the `hard'
constraints of the ALP with `soft' constraints in the form of a hinge-loss
function.

The balance of the
paper is concerned with establishing that the SALP forms the basis of a useful
approximate dynamic programming algorithm in large scale problems:
\begin{itemize}
\item We identify a concrete choice of violation budget $\theta$ and an
  idealized constraint violation distribution $\pi$ for which the SALP
  provides a useful relaxation in that the optimal solution can be a better
  approximation to the optimal cost-to-go function. This brings the cartoon
  improvement in Figure~\ref{fig:alp-salp} to fruition for general problems.
\item We show that the SALP is tractable (i.e., it is well approximated by an
  appropriate `sampled' version) and present computational experiments for a
  hard problem (Tetris) illustrating an order of magnitude improvement over
  the ALP.
\end{itemize}

\section{Analysis}\label{sec:analysis}

This section is dedicated to a theoretical analysis of the SALP. The
overarching objective of this analysis is to provide some assurance of the
soundness of the proposed approach. In some instances, the bounds we provide
will be directly comparable to bounds that have been developed for the ALP
method. As such, a relative consideration of the bounds in these two cases can
provide a theoretical comparison between the ALP and SALP methods. In
addition, our analysis will serve as a crucial guide to practical
implementation of the SALP as will be described in Section~\ref{sec:prac-imp}.
In particular, the theoretical analysis presented here provides intuition as
to how to select parameters such as the state-relevance weights and the
constraint violation distribution. We note that all of our bounds are relative to a measure of how well the approximation architecture employed is capable of approximating the optimal cost-to-go function; it is unreasonable to expect non-trivial bounds that are independent of the architecture used.


Our analysis will present three types of results:
\begin{itemize}
\item Approximation guarantees (Sections~\ref{sec:simple_bound} and
  \ref{sec:approx_gur}): We establish bounds on the distance between
  approximations computed by the SALP and the optimal value function $J^*$,
  relative to the distance between the best possible approximation afforded by
  the chosen basis functions and $J^*$. These guarantees will indicate that
  the SALP computes approximations that are of comparable quality to the
  projection\footnote{ Note that it is intractable to directly compute the
    projection since $J^*$ is unknown.
  } of $J^*$ on to the linear span of $\Phi$. 
\item Performance bounds (Section~\ref{sec:perf_bound}): While it is desirable
  to approximate $J^*$ as closely as possible, an important concern is the
  quality of the policies generated by acting greedily according to such
  approximations, as measured by their performance. We present bounds on the
  performance loss incurred, relative to the optimal policy, in using an SALP
  approximation.
\item Sample complexity results (Section~\ref{sec:sampling}): The SALP is a
  linear program with a large number of constraints as well as variables. In
  practical implementations, one may consider a `sampled' version of this
  program that has a manageable number of variables and constraints. We
  present sample complexity guarantees that establish bounds on the number of
  samples required to produce a good approximation to the solution of the
  SALP. These bounds scale linearly with the number of basis function $K$ and
  are independent of the size of the state space $\Xscr$.
\end{itemize}

\subsection{Idealized Assumptions}

Given the broad scope of problems addressed by ADP algorithms, analyses of
such algorithms typically rely on an `idealized' assumption of some sort. In
the case of the ALP, one either assumes the ability to solve a linear program
with as many constraints as there are states, or, absent that, knowledge of a
certain idealized sampling distribution, so that one can then proceed with
solving a `sampled' version of the ALP. Our analysis of the SALP in this
section is predicated on the knowledge of this same idealized sampling
distribution. In particular, letting $\mu^*$ be an optimal policy and
$P_{\mu^*}$ the associated transition matrix, we will require access to
samples drawn according to the distribution $\pi_{\mu^*,\nu}$ given by
\begin{equation}\label{eq:pimustar}
\pi_{\mu^*,\nu}^\top \defeq (1 - \alpha) \nu^\top(I-\alpha P_{\mu^*})^{-1}
=
(1-\alpha) \sum_{t=0}^\infty \alpha^t \nu^\top P_{\mu^*}^t.
\end{equation}
Here $\nu$ is an arbitrary initial distribution over states. The distribution
$\pi_{\mu^*,\nu}$ may be interpreted as yielding the discounted expected
frequency of visits to a given state when the initial state is distributed
according to $\nu$ and the system runs under the policy $\mu^*$. We note that
the `sampled' ALP introduced by \citet{deFarias2} requires access to states
sampled according to precisely this distribution. Theoretical analyses of
other approaches to approximate DP such as approximate value iteration and
temporal difference learning similarly rely on the knowledge of specialized
sampling distributions that cannot be obtained tractably
\citep[see][]{deFarias00}.

\subsection{A Simple Approximation Guarantee}\label{sec:simple_bound}

This section presents a first, simple approximation guarantee for the
following specialization of the SALP in \eqref{eq:salp},
\begin{equation}\label{eq:salp3}
  \begin{array}{lll}
    \maximize_{r, s} & \nu^\top \Phi r \\
    \subjectto & \Phi r \leq T\Phi r + s, \\
    & \pi_{\mu^*,\nu}^\top s \leq \theta,\quad s \geq 0.
  \end{array}
\end{equation}
Here, the constraint violation distribution is set to be $\pi_{\mu^*,\nu}$.

Before we state our approximation guarantee, consider the following function:
\begin{equation}\label{eq:flp}
  \begin{array}{llll}
    \ell(r,\theta) \defeq
    & \minimize_{s,\gamma} & \gamma/(1 - \alpha) \\
    & \subjectto & \Phi r - T\Phi r  \leq s + \gamma \1,
    \\
    & & \pi_{\mu^*,\nu}^\top s \leq \theta,\quad s \geq 0.
  \end{array}
\end{equation}
Suppose we are given a vector $r$ of basis function weights and a violation
budget $\theta$. As we will shortly demonstrate, $\ell(r,\theta)$ defines the
minimum translation (in the direction of the vector $\1$) of $r$ such so as to
get a feasible solution for \eqref{eq:salp3}. We will denote by $s(r,\theta)$
the $s$ component of the solution to \eqref{eq:flp}. The following lemma,
whose proof may be found in Appendix~\ref{sec:ell-proof}, characterizes the
function $\ell(r,\theta)$:

\begin{lemma}\label{le:ell}
  For any $r \in \R^K$ and $\theta \geq 0$:
  \begin{enumerate}[(i)]
  \item $\ell(r,\theta)$ is a finite-valued, decreasing, piecewise linear,
    convex function of $\theta$.
  \item\label{le:ell:2}
    \[
    \ell(r,\theta) \leq \frac{1+\alpha}{1-\alpha} \|  J^* - \Phi r\|_\infty.
    \]
  \item\label{le:ell:3} 
    The right partial derivative of $\ell(r,\theta)$ with
    respect to $\theta$ satisfies
    \[
    \frac{\partial^+}{\partial \theta^+} \ell(r,0) 
    = - \left((1 - \alpha) \sum_{x\in\Omega(r)} \pi_{\mu^*,\nu}(x)\right)^{-1},
    \]
    where 
    \[
    \Omega(r) \defeq \argmax_{x\in\Xscr}\ \Phi r(x) - T\Phi r (x).
    \]
  \end{enumerate}
\end{lemma}

Armed with this definition, we are now in a position to state our first, crude
approximation guarantee:
\begin{theorem}
\label{th:bound1}
Suppose that $r_{\text{SALP}}$ is an optimal solution to the SALP
\eqref{eq:salp3}, and let $r^*$ satisfy
\[
r^* \in \argmin_r\ \| J^* - \Phi r\|_\infty.
\]
Then,
\begin{equation}\label{eq:bound1}
\|J^* - \Phi r_{\text{SALP}} \|_{1,\nu}
\leq
\|J^* - \Phi r^* \|_{\infty} + \ell(r^*,\theta) + \frac{2\theta}{1-\alpha}.
\end{equation}
\end{theorem}

The above theorem allows us to interpret $\ell(r^*,\theta) +
2\theta/(1-\alpha)$ as an upper bound to the approximation error (in the
$\|\cdot\|_{1,\nu}$ norm) associated with the SALP solution $r_{\text{SALP}}$,
relative to the error of the \emph{best} approximation $r^*$ (in the
$\|\cdot\|_\infty$ norm).  This theorem also provides justification for the
intuition, described in Section~\ref{sec:salp}, that a relaxation of the
feasible region of the ALP will result in better value function
approximations.  To see this, consider the following corollary:
\begin{corollary}\label{co:U}
  Define $U_{\text{SALP}}(\theta)$ to be the upper bound
  in \eqref{eq:bound1}, i.e.,
  \[
  U_{\text{SALP}}(\theta)
  \defeq
  \|J^* - \Phi r^* \|_{\infty} + \ell(r^*,\theta) + \frac{2\theta}{1-\alpha}.
  \]
  Then:
  \begin{enumerate}[(i)]
  \item\label{co:U:1}
    \[
    U_{\text{SALP}}(0) 
    \leq 
    \frac{2}{1 - \alpha} \|J^* - \Phi r^* \|_{\infty}.
    \]
  \item\label{co:U:2} 
    The right partial derivative of $U_{\text{SALP}}(\theta)$ with
    respect to $\theta$ satisfies
    \[
    \frac{d^+}{d \theta^+} U_{\text{SALP}}(0) 
    = 
    \frac{1}{1-\alpha}
    \left[
      2
      - \left(\sum_{x\in\Omega(r^*)} \pi_{\mu^*,\nu}(x)\right)^{-1}
    \right].
    \]
  \end{enumerate}
\end{corollary}
\begin{proof}
  The result follows immediately from Parts~(\ref{le:ell:2}) and
  (\ref{le:ell:3}) of Lemma~\ref{le:ell}.
\end{proof}

Suppose that $\theta = 0$, in which case the SALP \eqref{eq:salp3} is
identical to the ALP \eqref{eq:alp}, thus,
$r_{\text{SALP}}=r_{\text{ALP}}$. Applying Part~(\ref{co:U:1}) of
Corollary~\ref{co:U}, we have, for the ALP, the approximation error bound
\begin{equation}\label{eq:bvr-bound1}
\|J^* - \Phi r_{\text{ALP}} \|_{1,\nu}
\leq
\frac{2}{1 - \alpha} \|J^* - \Phi r^* \|_{\infty}.
\end{equation}
This is precisely Theorem~2 of \citet{deFarias1}; we recover their
approximation guarantee for the ALP.

Now observe that, from Part~(\ref{co:U:2}) of Corollary~\ref{co:U}, if the set
$\Omega(r^*)$ is of very small probability according to the distribution
$\pi_{\mu^*,\nu}$, we expect that the upper bound $U_{\text{SALP}}(\theta)$
will decrease dramatically as $\theta$ is increased from $0$.\footnote{Already
  if $\pi_{\mu^*,\nu}(\Omega(r^*)) < 1/2$, $\frac{d^+}{d \theta^+}
  U_{\text{SALP}}(0) < 0.$} In other words, if the Bellman error $\Phi r^*(x)
- T\Phi r^*(x)$ produced by $r^*$ is maximized at states $x$ that are
collectively of small very probability, then we expect to have a choice of
$\theta>0$ for which
\[
U_{\text{SALP}}(\theta) 
\ll 
U_{\text{SALP}}(0)
\leq
\frac{2}{1 - \alpha} \|J^* - \Phi r^* \|_{\infty}.
\]
In this case, the bound \eqref{eq:bound1}
on the SALP solution will be an improvement over the bound
\eqref{eq:bvr-bound1} on the ALP solution.

Before we present the proof of Theorem~\ref{th:bound1} we present an auxiliary
claim that we will have several opportunities to use. The proof can be found
in Appendix~\ref{sec:ell-proof}.
\begin{lemma}\label{le:feas}
  Suppose that the vectors $J\in\R^\Xscr$ and $s\in\R^\Xscr$ satisfy
  \[
  J \leq T_{\mu^*} J + s.
  \]
  Then,
  \[
  J  \leq J^* +  \Delta^* s,
  \]
  where
  \[
  \Delta^* \defeq
  \sum_{k=0}^\infty (\alpha P_{\mu^*})^k = (I - \alpha P_{\mu^*})^{-1},
  \]
  and $P_{\mu^*}$ is the transition probability matrix corresponding to an
  optimal policy.

  In particular, if $(r,s)$ is feasible for the LP
  \eqref{eq:salp3}. Then,
  \[
  \Phi r  \leq J^* +  \Delta^* s.
  \]
\end{lemma}
A feasible solution to the ALP is necessarily a lower bound to the optimal
cost-to-go function, $J^*$. This is no longer the case for the SALP; the above
lemma characterizes the extent to which this restriction is relaxed. 

We now proceed with the proof of Theorem~\ref{th:bound1}:
\begin{proof}[\normalfont\sffamily\bfseries Proof of Theorem~\ref{th:bound1}]
First, define the weight vector $\tilde{r}\in\R^m$ by
\[
\Phi \tilde{r} = \Phi r^* - \ell(r^*,\theta)\1,
\]
and set $\tilde{s} = s(r^*,\theta)$, the $s$-component of the solution to the LP \eqref{eq:flp} with parameters $r^*$ and $\theta$. We will demonstrate that $(\tilde{r},\tilde{s})$ is feasible for \eqref{eq:salp}.
Observe that, by the definition of the LP \eqref{eq:flp},
\[
\Phi r^* \leq T \Phi r^* + \tilde{s} + (1-\alpha) \ell(r^*,\theta) \1.
\]
Then,
\[
\begin{split}
T \Phi \tilde{r} & = T \Phi r^*
- \alpha \ell(r^*,\theta) \1 \\
& \geq \Phi r^* - \tilde{s} - (1-\alpha) \ell(r^*,\theta) \1
- \alpha \ell(r^*,\theta) \1 \\
& = \Phi \tilde{r} - \tilde{s}.
\end{split}
\]
Now, let $(r_{\text{SALP}},\bar{s})$ be the solution to the SALP
\eqref{eq:salp3}. By Lemma~\ref{le:feas},
\[
\begin{split}
  \| J^* - \Phi r_{\text{SALP}} \|_{1,\nu}
  & \leq
  \| J^* - \Phi r_{\text{SALP}} + \Delta^* \bar{s} \|_{1,\nu}
  +
  \| \Delta^* \bar{s} \|_{1,\nu}
  \\
  & = \nu^\top (J^* - \Phi r_{\text{SALP}} + \Delta^* \bar{s})
  + \nu^\top \Delta^* \bar{s}
  \\
  & = \nu^\top (J^* - \Phi r_{\text{SALP}})
  + \frac{2 \pi_{\mu^*,\nu}^\top \bar{s}}{1 - \alpha}
  \\
  & \leq \nu^\top (J^* - \Phi r_{\text{SALP}})
  + \frac{2 \theta}{1 - \alpha}
  \\
  & \leq \nu^\top (J^* - \Phi \tilde{r})
  + \frac{2 \theta}{1 - \alpha}
  \\
  & \leq \|J^* - \Phi \tilde{r}\|_\infty
  + \frac{2 \theta}{1 - \alpha}
  \\
  & \leq \|J^* - \Phi r^*\|_\infty
  + \|\Phi r^* - \Phi\tilde{r}\|_\infty
  + \frac{2 \theta}{1 - \alpha}
  \\
  & =  
  \|J^* - \Phi r^*\|_\infty + \ell(r^*,\theta) + \frac{2 \theta}{1 - \alpha},
\end{split}
\]
as desired.
\end{proof}

While Theorem~\ref{th:bound1} reinforces the intuition (shown via
Figure~\ref{fig:alp-salp}) that the SALP will permit closer approximations to
$J^*$ than the ALP, the bound leaves room for improvement:
\begin{enumerate}
\item \label{pt1} The right hand side of our bound measures projection error,
  $\|J^* - \Phi r^*\|_{\infty}$ in the $\|\cdot\|_\infty$ norm. Since it is
  unlikely that the basis functions $\Phi$ will provide a uniformly good
  approximation over the entire state space, the right hand side of our bound
  could be quite large.
\item \label{pt2} As suggested by \eqref{eq:state_rel_wts}, the choice of
  state relevance weights can significantly influence the solution. In
  particular, it allows us to choose regions of the state space where we would
  like a better approximation of $J^*$. The right hand side of our bound,
  however, is independent of $\nu$.
\item \label{pt3} Our guarantee does not suggest a concrete choice of the
  violation budget parameter $\theta$.
\end{enumerate}

The next section will present a substantially refined approximation bound,
that will address these issues.

\subsection{A Stronger Approximation Guarantee} \label{sec:approx_gur}

With the intent of deriving stronger approximation guarantees, we begin this
section by introducing a `nicer' measure of the quality of approximation
afforded by $\Phi$. In particular, instead of measuring the approximation
error $J^* - \Phi r^*$ in the $\|\cdot\|_{\infty}$ norm as we did for our
previous bounds, we will use a weighted max norm defined according to:
\[
\|J\|_{\infty, 1/\psi}  \defeq \max_{x \in \Xscr}\ \frac{|J(x)|}{\psi(x)}.
\]
Here, $\psi\colon \Xscr \rightarrow [1,\infty)$ is a given `weighting'
function. The weighting function $\psi$ allows us to weight approximation
error in a non-uniform fashion across the state space and in this manner
potentially ignore approximation quality in regions of the state space that
are less relevant. We define $\Psi$ to be the set of all weighting functions,
i.e.,
\[
\Psi \defeq \left\{ \psi \in \R^\Xscr\ :\ \psi \geq \1\right\}.
\]
Given a particular $\psi \in \Psi$, we define a scalar
\[
\beta(\psi) \defeq \max_{x,a}\ 
\left|\frac{\sum_{x'} P_a(x,x') \psi(x')}{\psi(x)} \right|. 
\]
One may view $\beta(\psi)$ as a measure of the `stability' of the system. 

In addition to specifying the sampling distribution $\pi$, as we did in
Section~\ref{sec:simple_bound}, we will specify (implicitly) a particular
choice of the violation budget $\theta$. In particular, we will consider
solving the following SALP:
\begin{equation}\label{eq:salp4}
  \begin{array}{lll}
    \maximize_{r,s} & \nu^\top \Phi r - \frac{2\pi_{\mu^*,\nu}^\top s}{1-\alpha}\\
    \subjectto & \Phi r \leq T\Phi r + s,\quad s \geq 0. 
  \end{array}
\end{equation}
It is clear that \eqref{eq:salp4} is equivalent to \eqref{eq:salp3} for a
specific choice of $\theta$. We then have:
\begin{theorem}\label{th:smooth_bound}
  If $r_{\text{SALP}}$ is an optimal solution to \eqref{eq:salp4}, then
\[
\|J^* - \Phi r_{\text{SALP}}\|_{1,\nu}
\leq
\inf_{r,\psi \in \Psi}\ 
\| J^* - \Phi r\|_{\infty, \1/\psi}
\left(
\nu^\top \psi
+
\frac
{2(\pi_{\mu^*,\nu}^\top \psi + 1)(\alpha \beta(\psi) + 1)}
{1 - \alpha}
\right).
\]
\end{theorem}

Before presenting a proof for this approximation guarantee, it is worth
placing the result in context to understand its implications. For this, we
recall a closely related result shown by \citet{deFarias1} for the ALP. In
particular, \citet{deFarias1} showed that \emph{given} an appropriate
weighting function (in their context, a `Lyapunov' function) $\psi$, one may
solve an ALP, with $\psi$ in the span of the basis functions $\Phi$. The
solution $r_{\text{ALP}}$ to such an ALP then satisfies:
\begin{equation}\label{eq:alppsi}
\|J^* - \Phi r_{\text{ALP}}\|_{1,\nu} \leq \inf_r\ \| J^* - \Phi r\|_{\infty,
  \1/\psi} \frac{2 \nu^\top \psi}{1 - \alpha \beta(\psi)},
\end{equation}
provided that $\beta(\psi) < 1/\alpha$. Selecting an appropriate $\psi$ in
their context is viewed to be an important task for practical performance and
often requires a good deal of problem specific analysis; \citet{deFarias1}
identify appropriate $\psi$ for several queueing models. Note that this is
equivalent to identifying a desirable basis function. In contrast, the
guarantee we present optimizes over \emph{all possible} $\psi$ (including
those $\psi$ that do not satisfy the Lyapunov condition $\beta(\psi) <
1/\alpha$, and that are not necessarily in the span of $\Phi$). 

To make the comparison more precise, let us focus attention on a particular
choice of $\nu$, namely $\nu = \pi_{\mu^*} \triangleq \pi_*$, the stationary
distribution induced under an optimal policy $\mu^*$. In this case,
restricting attention to the set of weighting functions 
\[
\bar{\Psi} = \{\psi \in \Psi\ :\ 
\alpha \beta(\psi) < 1\},
\]
so as to make the two bounds comparable, Theorem~\ref{th:smooth_bound}
guarantees that
\[
\begin{split}
\|J^* - \Phi r_{\text{SALP}}\|_{1,\nu}
& \leq
\inf_{r,\psi\in\bar{\Psi}}\ 
\| J^* - \Phi r\|_{\infty, \1/\psi}
\left(
\pi_*^\top \psi
+
\frac
{2(\pi_*^\top \psi + 1)(\alpha \beta(\psi) + 1)}
{1 - \alpha}
\right)
\\
& \leq
\inf_{r,\psi\in\bar{\Psi}}\ 
\| J^* - \Phi r\|_{\infty, \1/\psi}
\frac
{9\pi_*^\top \psi}
{1 - \alpha}.
\end{split}
\]
On the other hand, observing that $\beta(\psi) \geq 1$ for all $\psi\in\Psi$,
the right hand side for the ALP bound \eqref{eq:alppsi} is at least
\[
\inf_{r}\ 
\| J^* - \Phi r\|_{\infty, \1/\psi}
\frac
{2 \pi_*^\top \psi}
{1 - \alpha}.
\]
Thus, the approximation guarantee of Theorem~\ref{th:smooth_bound}
allows us to view the SALP as \emph{automating} the critical procedure of
identifying a good Lyapunov function for a given problem.

\begin{proof}[\normalfont\sffamily\bfseries Proof of
  Theorem~\ref{th:smooth_bound}]

  Let $r \in \R^m$ and $\psi\in\Psi$ be arbitrary. Define the vectors
  $\tilde{\epsilon}, \tilde{s}\in\R^\Xscr$ component-wise by
  \begin{align*}
    \tilde{\epsilon}(x) & \defeq ((\Phi r) (x) - (T \Phi r) (x))^+, \\
    \tilde{s}(x) & \defeq \tilde{\epsilon}(x)\left(1 - \frac{1}{\psi(x)}\right).
  \end{align*}
  Notice that $0 \leq \tilde{s} \leq \tilde{\epsilon}$.

  We next make a few observations. First, define $\tilde{r}$ according to
  \[
  \Phi\tilde{r} = \Phi r - \frac{\| \tilde{\epsilon} \|_{\infty,1/\psi}}{1-\alpha} \1.
  \]
  Observe that, by a similar construction to that in Theorem~\ref{th:bound1},
  $(\tilde{r},\tilde{s})$ is feasible for \eqref{eq:salp4}. Also,
  \[
  \| \Phi r - \Phi \tilde{r} \|_{\infty}
  =
  \frac{\|\tilde{\epsilon} \|_{\infty,1/\psi}}{1-\alpha}
  \leq
  \frac{\|T\Phi r - \Phi r \|_{\infty,1/\psi}}{1-\alpha}.
  \]
  Furthermore,
  \[
  \begin{split}
    \pi_{\mu^*,\nu}^\top \tilde{s}
    &
    =
    \sum_{x\in\Xscr} \pi_{\mu^*,\nu}(x)\tilde{\epsilon}(x)(1 - 1/\psi(x))
    \\
    &
    \leq
    \pi_{\mu^*,\nu}^\top\tilde{\epsilon}
    \\
    &
    \leq
    (\pi_{\mu^*,\nu}^\top \psi) \| \tilde{\epsilon}\|_{\infty,\1/\psi}
    \\
    &
    \leq
    (\pi_{\mu^*,\nu}^\top \psi) \| T\Phi r - \Phi r\|_{\infty,\1/\psi}.
  \end{split}
  \]
  Finally, note that
  \[
  \nu^\top( J^* - \Phi r)
  \leq
  (\nu^\top \psi )\| J^* - \Phi r\|_{\infty,\1/\psi}.
  \]

  Now, suppose that $(r_{\text{SALP}},\bar{s})$ is an optimal solution to the
  SALP \eqref{eq:salp4}.  We have from the last set of inequalities in the
  proof of Theorem~\ref{th:bound1} and the above observations,
  \begin{equation}\label{eq:approx-bound}
  \begin{split}
    \|J^* - \Phi r_{\text{SALP}}\|_{1,\nu} & \leq \nu^\top(J^* - \Phi
    r_{\text{SALP}}) + \frac{2\pi^\top_{\mu^*,\nu}\bar{s}}{1-\alpha}
    \\
    & \leq \nu^\top(J^* - \Phi \tilde{r}) +
    \frac{2\pi^\top_{\mu^*,\nu}\tilde{s}}{1-\alpha}
    \\
    & \leq \nu^\top(J^* - \Phi r ) + \nu^\top(\Phi r - \Phi \tilde{r}) +
    \frac{2\pi^\top_{\mu^*,\nu}\tilde{s}}{1-\alpha}
    \\
    & \leq \nu^\top(J^* - \Phi r )+ \|\Phi r - \Phi \tilde{r}\|_{\infty} +
    \frac{2\pi^\top_{\mu^*,\nu}\tilde{s}}{1-\alpha}
    \\
    & 
    \leq (\nu^\top \psi )\| J^* - \Phi r\|_{\infty,\1/\psi} 
    + \frac{\|T\Phi r - \Phi r \|_{\infty,1/\psi}}{1-\alpha} 
    \left( 1 + 2\pi_{\mu^*,\nu}^\top\psi\right).
\end{split}
\end{equation}

Since our choice of $r$ and $\psi$ were arbitrary, we have: 
\begin{equation}
\label{eq:bellman_error}
\|J^* - \Phi r_{\text{SALP}}\|_{1,\nu}
\leq
\inf_{r,\psi \in \Psi}\ 
(\nu^\top \psi )\| J^* - \Phi r\|_{\infty,\1/\psi} 
+ \frac{\|T\Phi r - \Phi r \|_{\infty,1/\psi}}{1-\alpha} 
\left( 1 + 2\pi_{\mu^*,\nu}^\top\psi\right).
\end{equation}
We would like to relate the Bellman error term $T \Phi r - \Phi r$ on the right
hand side of \eqref{eq:bellman_error} to the approximation error $J^* - \Phi
r$. In order to do so, note that for any vectors $J,\bar{J}\in\R^\Xscr$,
\[
|TJ(x) - T\bar{J}(x)| \leq \alpha \max_{a\in\Ascr}\ 
\sum_{x'\in\Xscr} P_a(x,x') | J(x') - \bar{J}(x') |.
\]
Therefore,
\[
\begin{split}
\| T\Phi r - J^*\|_{\infty,\1/\psi}
& \leq 
\alpha \max_{x,a}\ 
\frac{\sum_{x'} P_a(x,x') | \Phi r(x') - J^*(x') |}{\psi(x)} 
\\
& \leq 
\alpha \max_{x,a}\ 
\frac{\sum_{x'} P_a(x,x')\psi(x') \frac{| \Phi r(x') - J^*(x') |}{\psi(x')}
}{\psi(x)} 
\\
& \leq
\alpha \beta(\psi) \| J^* - \Phi r \|_{\infty,1/\psi}.
\end{split}
\]
Thus,
\begin{equation}\label{eq:int1}
\begin{split}
\| T\Phi r - \Phi r\|_{\infty,\1/\psi}
&
\leq
\| T\Phi r - J^*\|_{\infty,\1/\psi} + \| J^* - \Phi r\|_{\infty,\1/\psi}
\\
&
\leq
 \| J^* - \Phi r \|_{\infty,1/\psi} (1 + \alpha \beta(\psi)).
\end{split}
\end{equation}
Combining \eqref{eq:bellman_error} and \eqref{eq:int1}, we get the desired
result.
\end{proof}

The analytical results provided in Sections~\ref{sec:simple_bound} and
\ref{sec:approx_gur} provide bounds on the quality of the approximation
provided by the SALP solution to $J^*$. The next section presents performance
bounds with the intent of understanding the increase in expected cost incurred
in using a control policy that is greedy with respect to the SALP
approximation in lieu of the optimal policy.

\subsection{A Performance Bound} \label{sec:perf_bound}

We will momentarily present a result that will allow us to interpret the objective of the SALP \eqref{eq:salp4} as an upper bound on the performance loss of a greedy policy with respect to the SALP solution. Before doing so, we briefly introduce some relevant notation. For a given policy $\mu$, we denote
\[
\Delta_\mu
\defeq
\sum_{k=0}^\infty (\alpha P_{\mu})^k 
=
(I - \alpha P_\mu)^{-1}.
\]
Thus, $\Delta^* = \Delta_{\mu^*}$. Given a vector $J\in\R^\Xscr$, let $\mu_J$
denote the greedy policy with respect to $J$. That is, $\mu_J$ satisfies
$T_{\mu_J} J = TJ$. Recall that the policy of interest to us will be
$\mu_{\Phi r_{\text{SALP}}}$ for a solution $r_{\text{SALP}}$ to the
SALP. Finally, for an arbitrary starting distribution over states $\eta$, we
define the `discounted' steady state distribution over states induced by
$\mu_J$ according to
\[
\nu(\eta,J)^\top 
\defeq
(1 - \alpha) \eta^\top \sum_{k=0}^\infty (\alpha P_{\mu_J})^k 
=
 (1-\alpha)\eta^\top \Delta_{\mu_J}.
\]

We have the following upper bound on the increase in cost incurred by using
$\mu_J$ in place of $\mu^*$:
\begin{theorem}
\label{th:performance_bound}
\[
\|J_{\mu_J} - J^*\|_{1,\eta}
\leq
\frac{1}{1-\alpha}
\left(
\nu(\eta,J)^\top(J^* - J) 
+ \frac{2}{1-\alpha} \pi^\top_{\mu^*,\nu(\eta,J)} (J - TJ)^+
\right).
\]
\end{theorem}
Theorem~\ref{th:performance_bound} indicates that if $J$ is close to $J^*$, so
that $(J - TJ)^+$ is also small, then the expected cost incurred in using a
control policy that is greedy with respect to $J$ will be close to
optimal. The bound indicates the impact of approximation errors over differing
parts of the state space on performance loss.

Suppose that $(r_{\text{SALP}},\bar{s})$ is an optimal solution to the SALP
\eqref{eq:salp4}. Then, examining the proof of Theorem~\ref{th:smooth_bound}
and, in particular, \eqref{eq:approx-bound}, reveals that
\begin{equation}\label{eq:approx-bound2}
\begin{split}
\lefteqn{
\nu^\top(J^* - \Phi r_{\text{SALP}}) 
+ \frac{2}{1-\alpha} \pi^\top_{\mu^*,\nu}
\bar{s}
}
\\
& \leq
\inf_{r,\psi \in \Psi}\ 
\| J^* - \Phi r\|_{\infty, \1/\psi}
\left(
\nu^\top \psi
+
\frac
{2(\pi_{\mu^*,\nu}^\top \psi + 1)(\alpha \beta(\psi) + 1)}
{1 - \alpha}
\right).
\end{split}
\end{equation}
Assume that the state relevance weights $\nu$ in the SALP \eqref{eq:salp4}
satisfy 
\begin{equation}\label{eq:nufixed}
  \nu = \nu(\eta,\Phi r_{\text{SALP}}).
\end{equation}
Then, combining Theorem~\ref{th:smooth_bound}  and
\eqref{eq:approx-bound2} yields
\begin{equation}\label{eq:approx-bound3}
\|J_{\mu_{\Phi r_{\text{SALP}}}} - J^*\|_{1,\eta}
\leq
\frac{1}{1-\alpha}\left(
\inf_{r,\psi \in \Psi}\ 
\| J^* - \Phi r\|_{\infty, \1/\psi}
\left(
\nu^\top \psi
+
\frac
{2(\pi_{\mu^*,\nu}^\top \psi + 1)(\alpha \beta(\psi) + 1)}
{1 - \alpha}
\right)\right).
\end{equation}
This bound \emph{directly} relates the performance loss of the SALP policy to
the ability of the basis function architecture $\Phi$ to approximate
$J^*$. Moreover, this relationship allows us to loosely interpret the SALP as
minimizing an upper bound on performance loss. 

Unfortunately, it is not clear how to make an a-priori choice of the state
relevance weights $\nu$ to satisfy \eqref{eq:nufixed}, since the choice of
$\nu$ determines the solution to the SALP $ r_{\text{SALP}}$; this is
essentially the situation one faces in performance analyses for approximate
dynamic programming algorithms such as approximate value iteration and
temporal difference learning \citep{deFarias00}. Indeed, it is not clear
that there exists a $\nu$ that solves the fixed point equation
\eqref{eq:nufixed}. On the other hand, given a choice of $\nu$ so that $\nu
\approx \nu(\eta,\Phi r_{\text{SALP}})$, in the sense of a small Radon-Nikodym
derivative between the two distributions, an approximate version of the
performance bound \eqref{eq:approx-bound3} will hold. As suggested by
\cite{deFarias1} in the ALP case, one possibility for finding such a choice of
state relevance weights is to iteratively resolve the SALP, and at each time
using the policy from the prior iteration to generate state relevance weights
for the next iteration.



\begin{proof}[\normalfont\sffamily\bfseries
  Proof of Theorem~\ref{th:performance_bound}]
  Define $s \defeq (J - TJ)^+$. From Lemma~\ref{le:feas}, we know that
\[
J \leq J^* + \Delta^*s.
\]
Applying $T_{\mu^*}$ to both sides,
\[
T_{\mu^*} J
\leq
J^* + \alpha P_{\mu^*} \Delta^*s 
 =
J^* + \Delta^*s - s 
\leq
J^* + \Delta^*s,
\]
so that
\begin{equation}
\label{eq:perf_fact1}
TJ \leq T_{\mu^*}J \leq J^* + \Delta^*s.
\end{equation}
Then,
\begin{equation}
\label{eq:perf_fact2}
\begin{split}
\eta^\top(J_{\mu_J} - J)
&
=
\eta^\top \sum_{k=0}^\infty \alpha^k P_{\mu_J}^k (g_\mu + \alpha P_{\mu_J}J - J)
\\
&
=
\eta^\top \Delta_{\mu_J}(TJ - J)
\\
&
\leq
\eta^\top\Delta_{\mu_J}(J^* - J + \Delta^*s)
\\
&
=
\frac{1}{1-\alpha}\nu(\eta,J)^\top(J^* - J + \Delta^*s).
\end{split}
\end{equation}
where the second equality is from the fact that $g_\mu + \alpha P_{\mu_J}J = T_{\mu_J}J = TJ$, and the inequality follows from \eqref{eq:perf_fact1}.

Further,
\begin{equation}
\label{eq:perf_fact3}
\begin{split}
\eta^\top(J - J^*)
&
\leq
\eta^\top \Delta^*s
\\
&
\leq
\eta^\top \Delta_{\mu_J}\Delta^*s
\\
&
=
\frac{1}{1-\alpha}\nu(\eta,J)^\top \Delta^*s.
\end{split}
\end{equation}
where the second inequality follows from the fact that $\Delta^*s \geq 0$ and $\Delta_{\mu_J} = I + \sum_{k=1}^\infty \alpha^k P_{\mu_J}^k$. 

It follows from \eqref{eq:perf_fact2} and \eqref{eq:perf_fact3} that
\[
\begin{split}
\eta^\top(J_{\mu_J} - J^*)
&=
\eta^\top(J_{\mu_J} - J)  + \eta^\top(J - J^*)
\\
&
\leq
\frac{1}{1-\alpha}\nu(\eta,J)^\top(J^* - J + 2 \Delta^*s)
\\
&
=
\frac{1}{1-\alpha}
\left(
\nu(\eta,J)^\top(J^* - J) + \frac{2}{1-\alpha} \pi^\top_{\mu^*,\nu(\eta,J)}s
\right),
\end{split}
\]
which is the result.
\end{proof}

\subsection{Sample Complexity}\label{sec:sampling}

Our analysis thus far has assumed we have the ability to solve the SALP. The
number of constraints and variables in the SALP is grows linearly with
the size of the state space $\Xscr$. Hence, this program will typically be
intractable for problems of interest. One solution, which we describe here, is
to consider a \emph{sampled} variation of the SALP, where states and
constraints are sampled rather than exhaustively considered. In this section,
we will argue that the solution to the SALP is well approximated by the
solution to a tractable, sampled variation.

In particular, let $\hat{\Xscr}$ be a collection of $S$ states drawn
independently from the state space $\Xscr$ according to the distribution
$\pi_{\mu^*,\nu}$. Consider the following optimization program:
\begin{equation}\label{eq:salpsampled}
  \begin{array}{lll}
    \maximize_{r,s} & \displaystyle 
    \nu^{\top} \Phi r - \frac{2}{(1-\alpha) S} \sum_{x\in\hat{\Xscr}} s(x)
    \vspace{3pt}
    \\
    \subjectto & \Phi r (x) \leq T\Phi r (x) + s(x), & \forall\ x\in\hat{\Xscr},
    \\
    & s \geq 0,\quad r\in\Nscr.
  \end{array}
\end{equation}
Here, $\Nscr \subset \R^K$ is a bounding set that restricts the magnitude of
the sampled SALP solution, we will discuss the role of $\Nscr$ shortly. Notice
that \eqref{eq:salpsampled} is a variation of \eqref{eq:salp4}, where only the
decision variables and constraints corresponding to the sampled subset of
states are retained. The resulting optimization program has $K+S$ decision
variables and $S|\Ascr|$ linear constraints. For a moderate number of samples
$S$, this is easily solved.  Even in scenarios where the size of the action
space $\Ascr$ is large, it is frequently possible to rewrite
\eqref{eq:salpsampled} as a compact linear program \citep{FariasVanRoyNetRM,
  Moallemi08}. The natural question, however, is whether the solution to the
sampled SALP \eqref{eq:salpsampled} is a good approximation to the solution
provided by the SALP \eqref{eq:salp4}, for a `tractable' number of samples
$S$.

Here, we answer this question in the affirmative.  We will provide a sample
complexity bound that indicates that for a number of samples $S$ that scales
linearly with the dimension of $\Phi$, $K$, and that need not depend on the
size of the state space, the solution to the sampled SALP satisfies, with high
probability, the approximation guarantee presented for the SALP solution in
Theorem~\ref{th:smooth_bound}. 

Our proof will rely on the following lemma, which provides a Chernoff bound
for the \emph{uniform} convergence of a certain class of functions. The proof
of this lemma, which is based on bounding the pseudo-dimension of the
class of functions, can be found in Appendix~\ref{sec:sampling-proof}.
\begin{lemma} \label{le:sample_complexity} 
  Given a constant $B > 0$, define
  the function $\zeta\colon \R \rightarrow [0,B]$ by
  \[
  \zeta(t) \defeq \max\left(\min(t,B),0\right).
  \]
  Consider a pair of random variables $(Y,Z) \in\R^K\times\R$. For each
  $i=1,\ldots,n$, let the pair $\big(Y^{(i)},Z^{(i)}\big)$ be an i.i.d. sample
  drawn according to the distribution of $(Y,Z)$. Then, for all $\epsilon \in
  (0,B]$,
  \begin{multline*}
    \PR\left(
      \sup_{r\in\R^K}\  \left|
        \frac{1}{n} \sum_{i=1}^n \zeta\left(r^\top Y^{(i)} + Z^{(i)}\right)
        - \E\left[ \zeta\left(r^\top Y + Z\right) \right]
      \right|
      > \epsilon
    \right)
    \\
    \leq
    8\left(\frac{32eB}{\epsilon} \log \frac{32eB}{\epsilon}  \right)^{K+2} 
    \exp\left( - \frac{\epsilon^2n}{64B^2} \right).
  \end{multline*}
  Moreover, given $\delta \in (0,1)$, if
  \[
  n \geq \frac{64B^2}{\epsilon^2}\left(
    2(K+2)\log \frac{16eB}{\epsilon}+\log\frac{8}{\delta} 
  \right),
  \]
  then this probability is at most $\delta$.
\end{lemma}

In order to establish a sample complexity result, we require control over the
magnitude of optimal solutions to the SALP \eqref{eq:salp4}. This control is
provided by the bounding set $\Nscr$. In particular, we will assume that
$\Nscr$ is large enough so that it contains an optimal solution to the SALP
\eqref{eq:salp4}, and we define the constant
\begin{equation}\label{eq:B-def}
B \defeq \sup_{r\in\Nscr}\ \|(\Phi r - T \Phi r)^+\|_{\infty}.
\end{equation}
This quantity is closely related to the diameter of the region $\Nscr$. Our
main sample complexity result can then be stated as follows:
\begin{theorem}\label{th:sample_complexity_RSALP}
  Under the conditions of Theorem~\ref{th:smooth_bound}, let $r_{\text{SALP}}$
  be an optimal solution to the SALP \eqref{eq:salp4}, and let
  $\hat{r}_{\text{SALP}}$ be an optimal solution to the sampled SALP
  \eqref{eq:salpsampled}. Assume that $r_{\text{SALP}}\in\Nscr$. Further,
  given $\epsilon \in (0,B]$ and $\delta \in (0,1/2]$, suppose that the number
  of sampled states $S$ satisfies
  \[
  S \geq \frac{64B^2}{\epsilon^2}\left(2(K+2)\log \frac{16eB}{\epsilon}
    +\log\frac{8}{\delta} \right).
  \]
  Then, with probability at least $1-\delta-2^{-383} \delta^{128}$,
  \[
  \|J^* - \Phi \hat{r}_{\text{SALP}}\|_{1,\nu}
  \leq
  \inf_{\substack{r\in\Nscr \\ \psi \in \Psi}}\ 
  \| J^* - \Phi r\|_{\infty, \1/\psi}
  \left(
    \nu^\top \psi
    +
    \frac
    {2(\pi_{\mu^*,\nu}^\top \psi + 1)(\alpha \beta(\psi) + 1)}
    {1 - \alpha}
  \right)
  +
  \frac{4\epsilon}{1-\alpha}.
  \]
\end{theorem}

Theorem~\ref{th:sample_complexity_RSALP} establishes that the sampled SALP
\eqref{eq:salpsampled} provides a close approximation to the solution of the
SALP \eqref{eq:salp4}, in the sense that the approximation guarantee we
established for the SALP in Theorem~\ref{th:smooth_bound} is approximately
valid for the solution to the sampled SALP, with high probability. The theorem
precisely specifies the number of samples required to accomplish this
task. This number depends linearly on the number of basis functions and the
diameter of the feasible region, but is otherwise independent of the size of
the state space for the MDP under consideration.

It is worth juxtaposing our sample complexity result with that available for
the ALP \eqref{eq:alp}. Recall that the ALP has a large number of constraints
but a \emph{small} number of variables; the SALP is thus, at least
superficially, a significantly more complex program. Exploiting the fact that
the ALP has a small number of variables, \citet{deFarias2} establish a sample
complexity bound for a sampled version of the ALP analogous to the the sampled
SALP \eqref{eq:salpsampled}. The number of samples required for this sampled
ALP to produce a good approximation to the ALP can be shown to depend on the
same problem parameters we have identified here, viz.: the constant $B$ and
the number of basis functions $K$. The sample complexity in the ALP case is
identical to the sample complexity bound established here, up to constants and
a linear dependence on the ratio $B/\epsilon$. This is as opposed to the
quadratic dependence on $B/\epsilon$ of the sampled SALP. Although the two
sample complexity bounds are within polynomial terms of each other, one may
rightfully worry abut the practical implications of an additional factor of
$B/\epsilon$ in the required number of samples. In the computational study of
Section~\ref{sec:tetris}, we will attempt to address this concern.

Finally, note that the sampled SALP has $K + S$ variables and $S|\Ascr|$
linear constraints whereas the sampled ALP has merely $K$ variables and
$S|\Ascr|$ linear constraints. Nonetheless, we will show in the
Section~\ref{sec:efficient} that the special structure of the Hessian
associated with the sampled SALP affords us a linear computational complexity
dependence on $S$.

\begin{proof}[\normalfont\sffamily\bfseries Proof of
  Theorem~\ref{th:sample_complexity_RSALP}]
  Define the vectors
  \[
  \hat{s}_{\mu^*} \defeq \left(\Phi \hat{r}_{\text{SALP}}
    - T_{\mu^*} \Phi \hat{r}_{\text{SALP}} \right)^{+},
  \quad\text{and}\quad
  \hat{s} \defeq
  \left(\Phi \hat{r}_{\text{SALP}} - T \Phi \hat{r}_{\text{SALP}}
  \right)^{+}.
  \]
  One has, via Lemma~\ref{le:feas}, that
  \[
  \Phi \hat{r}_{\text{SALP}} - J^* \leq \Delta^*\hat{s}_{\mu^*}
  \]
  Thus, as in the last set of inequalities in the proof of
  Theorem~\ref{th:bound1}, we have
\begin{equation}\label{eq:Jhr}
\|J^* - \Phi \hat{r}_{\text{SALP}}\|_{1,\nu}
\leq
\nu^\top(J^* - \Phi \hat{r}_{\text{SALP}}) 
+ \frac{2\pi^\top_{\mu^*,\nu}\hat{s}_{\mu^*}}{1-\alpha}.
\end{equation}

Now, let $\hat{\pi}_{\mu^*,\nu}$ be the empirical measure induced by the
collection of sampled states $\hat{\Xscr}$. Given a state $x\in\Xscr$, define a vector $Y(x)\in\R^K$ and a scalar $Z(x)\in\R$ according to
\[
Y(x) \defeq \Phi(x)^\top - \alpha P_{\mu^*} \Phi(x)^\top,
\quad
Z(x)\defeq - g(x,\mu^*(x)),
\]
so that, for any vector of weights $r\in\Nscr$,
\[
\left(\Phi r(x) - T_{\mu^*}\Phi r(x)\right)^+ = \zeta\left(r^\top Y(x) 
  + Z(x)\right).
\]
Then,
\[
\left| \hat{\pi}^\top_{\mu^*,\nu}\hat{s}_{\mu^*} 
- \pi_{\mu^*,\nu}^\top\hat{s}_{\mu^*}
\right|
\leq 
\sup_{r\in\Nscr}\ 
\left|
\frac{1}{S} \sum_{x\in\hat{\Xscr}} \zeta\left(r^\top Y(x) + Z(x)\right)
-
\sum_{x\in\Xscr}
\pi_{\mu^*,\nu}(x) \zeta\left(r^\top Y(x) + Z(x)\right)
\right|.
\]
Applying Lemma~\ref{le:sample_complexity}, we have that
\begin{equation}
\label{eq:uniform_sample_bound}
\PR\left(
\left| \hat{\pi}^\top_{\mu^*,\nu}\hat{s}_{\mu^*} - \pi_{\mu^*,\nu}^\top\hat{s}_{\mu^*} 
\right| 
> \epsilon
\right)
\leq \delta.
\end{equation}

Next, suppose $(r_{\text{SALP}},\bar{s})$ is an optimal solution to the SALP
\eqref{eq:salp4}. Then, with probability at least $1-\delta$,
\begin{equation}\label{eq:hr1}
\begin{split}
\nu^\top(J^* - \Phi \hat{r}_{\text{SALP}}) 
+ \frac{2\pi^\top_{\mu^*,\nu}\hat{s}_{\mu^*}}{1-\alpha}
& 
\leq
\nu^\top(J^* - \Phi \hat{r}_{\text{SALP}}) 
+ \frac{2\hat{\pi}^\top_{\mu^*,\nu}\hat{s}_{\mu^*}}{1-\alpha}
+ \frac{2 \epsilon}{1-\alpha} 
\\
&
\leq
\nu^\top(J^* - \Phi \hat{r}_{\text{SALP}}) 
+ \frac{2\hat{\pi}^\top_{\mu^*,\nu}\hat{s}}{1-\alpha}
+ \frac{2 \epsilon}{1-\alpha}
\\
&
\leq
\nu^\top(J^* - \Phi r_{\text{SALP}}) 
+ \frac{2\hat{\pi}^\top_{\mu^*,\nu}\bar{s}}{1-\alpha}
+ \frac{2 \epsilon}{1-\alpha},
\end{split}
\end{equation}
where the first inequality follows from \eqref{eq:uniform_sample_bound}, and
the final inequality follows from the optimality of
$(\hat{r}_{\text{SALP}},\hat{s})$ for the sampled SALP \eqref{eq:salpsampled}.

Notice that, without loss of generality, we can assume that $\bar{s}(x)=(\Phi
r_{\text{SALP}}(x) - T \Phi r_{\text{SALP}}(x))^+$, for each
$x\in\Xscr$. Thus, $0 \leq \bar{s}(x) \leq B$. Applying Hoeffding's
inequality,
\[
\PR\left(
\left| \hat{\pi}^\top_{\mu^*,\nu}\bar{s} - \pi_{\mu^*,\nu}^\top\bar{s} 
\right| 
\geq \epsilon
\right)
\leq 2 \exp\left(-\frac{2S\epsilon^2}{B^2}\right) 
< 2^{-383} \delta^{128},
\]
where final inequality follows from our choice of $S$. Combining this with
\eqref{eq:Jhr} and \eqref{eq:hr1}, with probability at least $1-\delta - 2^{-383} \delta^{128}$, we
have
\[
\begin{split}
\|J^* - \Phi \hat{r}_{\text{SALP}}\|_{1,\nu}
&
\leq
\nu^\top(J^* - \Phi r_{\text{SALP}}) 
+ \frac{2\hat{\pi}^\top_{\mu^*,\nu}\bar{s}}{1-\alpha}
+ \frac{2 \epsilon}{1-\alpha}
\\
& \leq
\nu^\top(J^* - \Phi r_{\text{SALP}}) 
+ \frac{2\pi^\top_{\mu^*,\nu}\bar{s}}{1-\alpha}
+ \frac{4 \epsilon}{1-\alpha}.
\end{split}
\]
The result then follows from \eqref{eq:approx-bound}--\eqref{eq:int1} in the
proof of Theorem~\ref{th:smooth_bound}.
\end{proof}

An alternative sample complexity bound of a similar flavor can be developed
using results from the stochastic programming literature. The key idea is that
the SALP \eqref{eq:salp4} can be reformulated as the following convex
stochastic programming problem:
\begin{equation}\label{eq:stoch_prg}
  \maximize_{r \in \Nscr}\ 
\E_{\nu, \pi_{\mu^*,\nu}}\left[\Phi r(x_0)  
- \frac{2}{1-\alpha}(\Phi r(x) -T\Phi r (x))^+\right],
\end{equation}
where $x_0,x\in\Xscr$ have distributions $\nu$ and $\pi_{\mu^*,\nu}$,
respectively. Interpreting the sampled SALP \eqref{eq:salpsampled} as a sample
average approximation of \eqref{eq:stoch_prg}, a sample complexity bound can
be developed using the methodology of \citet[][Chap.~5]{Shapiro09}, for
example.  This proof is simpler than the one presented here, but yields a
cruder estimate that is not as easily compared with those available for the
ALP.


\section{Practical Implementation}\label{sec:prac-imp}

The SALP \eqref{eq:salp}, as it is written, is not directly implementable. As
discussed in Section~\ref{sec:sampling}, the number of variables and
constraints grows linearly with the size of the state space $\Xscr$, making
the optimization problem intractable. Moreover, it is not clear how to choose
parameters such as the probability distributions $\nu$ and $\pi$ or the
violation budget $\theta$. However, the analysis in Section~\ref{sec:analysis}
provides insight that allows us to codify a recipe for a practical and
implementable variation.

Consider the following algorithm:
\begin{enumerate}
\item Sample $S$ states independently from the state space $\Xscr$ according
  to a sampling distribution $\rho$. Denote the set of sampled states by
  $\hat{\Xscr}$.

\item\label{st:line-search} Perform a line search over increasing choices of
  $\theta\geq 0$. For each choice of $\theta$,
\begin{enumerate}
\item\label{st:rlp} Solve the {\em sampled} SALP:
\begin{equation}\label{eq:rlp}
\begin{array}{lll}
\maximize_{r, s} & 
\displaystyle
\frac{1}{S} \sum_{x \in \hat{\Xscr}}(\Phi r)(x) 
\vspace{3pt}
\\
\subjectto 
& \Phi r(x) \leq T\Phi r(x) + s(x), & \forall\ x\in\hat{\Xscr}, 
\\[5pt]
& 
\displaystyle
\frac{1}{S} \sum_{x \in \hat{\Xscr}} s(x) \leq  \theta, 
\vspace{3pt}
\\
& s \geq 0.
\end{array}
\end{equation}

\item Evaluate the performance of the policy resulting from \eqref{eq:rlp} via
  Monte Carlo simulation.
\end{enumerate}

\item Select the best of the policies evaluated in Step~\ref{st:line-search}.
\end{enumerate}
This algorithm takes as inputs the following parameters:
\begin{itemize}
\item $\Phi$, a collection of $K$ basis functions. 
\item $S$, the number of states to sample. By sampling $S$ states, we limit
  the number of variables and constraints in the sampled SALP
  \eqref{eq:rlp}. Thus, by keeping $S$ small, the sampled SALP becomes
  tractable to solve numerically. On the other hand, the quality of the
  approximation provided by the sampled SALP may suffer is $S$ is chosen to be
  too small. The sample complexity theory developed in
  Section~\ref{sec:sampling} suggests that $S$ can be chosen to grow linearly
  with $K$, the size of the basis set. In particular, a reasonable choice of
  $S$ need not depend on the size of the underlying state space.

  In practice, we choose $S \gg K$ to be as large as possible subject to
  limits on the CPU time and memory required to solve \eqref{eq:rlp}. In
  Section~\ref{sec:efficient}, we will discuss how the program \eqref{eq:rlp}
  can be solved efficiently via barrier methods for large choices of $S$.

\item $\rho$, a sampling distribution on the state space $\Xscr$.  The
  distribution $\rho$ is used, via Monte Carlo sampling, in place of both the
  distributions $\nu$ and $\pi$ in the SALP \eqref{eq:salp}.  Recall that the
  bounds in Theorems~\ref{th:bound1} and \ref{th:smooth_bound} provide
  approximation guarantees in a $\nu$-weighted 1-norm. This suggests that
  $\nu$ should be chosen to emphasize regions of the state space where the
  quality of approximation is most important.  Similarly, the theory in
  Section~\ref{sec:analysis} suggests that the distribution $\pi$ should be
  related to the distribution induced by the optimal policy.

  In practice, we choose $\rho$ to be the stationary distribution under a
  baseline policy. States are then sampled from $\rho$ via Monte Carlo
  simulation of the baseline policy. This baseline policy can correspond, for
  example, to a heuristic control policy for the system. More sophisticated
  procedures such as `bootstrapping' can also be considered
  \citep{FariasVanRoyTetris}. Here, one starts with a heuristic policy to be
  used for sampling states. Given the sampled states, the application of our
  algorithm will result in a new control policy. The new control policy can
  then be used for state sampling in a subsequent round of optimization, and
  the process can be repeated.
\end{itemize}

Note that our algorithm does not require an explicit choice of the violation
budget $\theta$, since we optimize with a line search over the choices of
$\theta$. This is motivated by the fact that the sampled SALP \eqref{eq:rlp}
can efficiently resolved for increasing values of $\theta$ via a
`warm-start' procedure. Here, the optimal solution of the sampled SALP given
previous value of $\theta$ is used as a starting point for the solver in a
subsequent round of optimization. Using this method we observe that, in
practice, the total solution time for a series of sampled SALP instances that
vary by their values of $\theta$ grows sub-linearly with the number of
instances.

\subsection{Efficient Linear Programming Solution}\label{sec:efficient}

The sampled SALP \eqref{eq:rlp} can be written explicitly in the form of a
linear program:
\begin{equation}\label{eq:lp-form}
\begin{array}{lll}
\maximize_{r, s} & 
\displaystyle
c^\top r
\\
\subjectto 
& 
\begin{bmatrix}
A_{11} & A_{12} \\
0 & d^\top
\end{bmatrix}
\begin{bmatrix}
r \\ s
\end{bmatrix}
\leq 
b, \\
& s \geq 0.
\end{array}
\end{equation}
Here, $b\in\R^{S|\Ascr|+1}$, $c\in\R^K$, and $d\in\R^S$ are vectors,
$A_{11}\in\R^{S|\Ascr|\times K}$ is a dense matrix, and
$A_{12}\in\R^{S|\Ascr|\times S}$ is a sparse matrix.  This LP has $K+S$
decision variables and $S |\Ascr| +1$ linear constraints.

Typically, the number of sampled states $S$ will be quite large. For example,
in Section~\ref{sec:tetris}, we will discuss an example where $K=22$ and
$S=300{,}000$. The resulting LP has approximately $300{,}000$ variables
and $6{,}600{,}000$ constraints. In such cases, with many variables
\emph{and} many constraints, one might expect the LP to be difficult to
solve. However, the sparsity structure of the constraint matrix in
\eqref{eq:lp-form} and, especially, that of the sub-matrix $A_{12}$, allows
efficient optimization of this LP.

In particular, imagine solving the LP \eqref{eq:lp-form} with a barrier
method. The computational bottleneck of such a method is the inner Newton step
to compute a central point \citep[see, for example,][]{BoydConvex}. This step
involves the solution of a system of linear equations of the form
\begin{equation}\label{eq:newton-update}
H
\begin{bmatrix}
\Delta r \\ \Delta s
\end{bmatrix}
= -g.
\end{equation}
Here, $g\in\R^{K+S}$ is a vector and $H\in\R^{(K+S)\times(K+S)}$ is the
Hessian matrix of the barrier function. Without exploiting the structure of
the matrix $H$, this linear system can be solved with $O((K+S)^3)$ floating
point operations. For large values of $S$, this may be prohibitive. 

Fortunately, the Hessian matrix $H$ can be decomposed according to the block
structure 
\[
H \defeq
\begin{bmatrix}
H_{11} & H_{12} \\
H_{12}^\top & H_{22}
\end{bmatrix},
\]
where $H_{11}\in\R^{K\times K}$, $H_{12}\in\R^{K\times S}$, and
$H_{22}\in\R^{S\times S}$. In the case of the LP \eqref{eq:lp-form}, it is not
difficult to see that the sparsity structure of the sub-matrix $A_{12}$
ensures that the sub-matrix $H_{22}$ takes the form of a diagonal matrix plus
a rank-one matrix. This allows the linear system \eqref{eq:newton-update} to
be solved with $O(K^2 S + K^3)$ floating point operations. This
is linear in $S$, the number of sampled states.


\section{Case Study: Tetris}\label{sec:tetris}

Tetris is a popular video game designed and developed by Alexey Pazhitnov in
1985. The Tetris board, illustrated in Figure~\ref{fig:tetris}, consists of a
two-dimensional grid of 20 rows and 10 columns. The game starts with an empty
grid and pieces fall randomly one after another. Each piece consists of four
blocks and the player can rotate and translate it in the plane before it
touches the `floor'. The pieces come in seven different shapes and the next
piece to fall is chosen from among these with equal probability. Whenever the pieces are placed
such that there is a line of contiguous blocks formed, a point is earned and
the line gets cleared. Once the board has enough blocks such that the incoming
piece cannot be placed for all translation and rotation, the game terminates.
Hence the goal of the player is to clear maximum number of lines before the
board gets full.

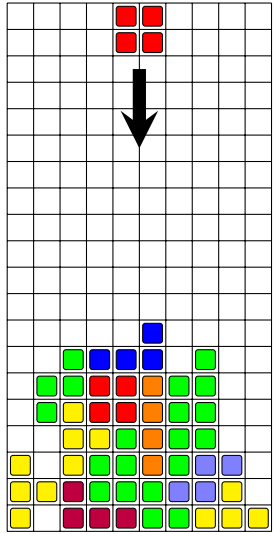
\begin{figure}[htb]
  \centering
  \hidefastcompile{
  \begin{tikzpicture}[x=10pt,y=10pt]
    \draw[shift={(-0.5,-0.5)},step=1] (0,0) grid (10,20);
    \tikzset{fsquare/.style={rectangle,rounded corners=1pt,draw=black,
        minimum size=7.5pt}}

    \foreach \pos in {(0,0),(7,0),(8,0),(9,0),(0,1),(1,1),(8,1),(0,2),(2,2),
      (2,3),(3,3),(2,4)}
    \node[fsquare,fill=yellow] at \pos {};

    \foreach \pos in {(2,0),(3,0),(4,0),(2,1)}
    \node[fsquare,fill=purple] at \pos {};

    \foreach \pos in {(3,4),(4,4),(3,5),(4,5))}
    \node[fsquare,fill=red] at \pos {};

    \foreach \pos in {(5,0),(6,0),(3,1),(4,1),(5,1),(3,2),(4,2),(6,2),(4,3),(6,3),
      (7,3),(1,4),(6,4),(7,4),(1,5),(2,5),(6,5),(7,5),(2,6),(7,6)}
    \node[fsquare,fill=green] at \pos {};

    \foreach \pos in {(6,1),(7,1),(7,2),(8,2)}
    \node[fsquare,fill=blue!50] at \pos {};

    \foreach \pos in {(5,2),(5,3),(5,4),(5,5)}
    \node[fsquare,fill=orange] at \pos {};

    \foreach \pos in {(3,6),(4,6),(5,6),(5,7)}
    \node[fsquare,fill=blue] at \pos {};

    \foreach \pos in {(4,18),(5,18),(4,19),(5,19)}
    \node[fsquare,fill=red] at \pos {};

    \draw[->,black,line width=5pt] (4.5,17) to (4.5,14);
  \end{tikzpicture}
  }
  \caption{Example of a Tetris board configuration\label{fig:tetris}}
\end{figure}

Our interest in Tetris as a case study for the SALP algorithm is motivated by
several facts. First, theoretical results suggest that design of an optimal
Tetris player is a difficult problem.  \citet{Brzustowski1992} and
\citet{Burgiel} have shown that the game of Tetris has to end with probability
one, under all policies. They demonstrate a sequence of pieces, which leads to
termination state of game for all possible actions.  \citet{Demaine} consider
the offline version of Tetris and provide computational complexity
results for `optimally' playing Tetris. They show that when the sequence of
pieces is known beforehand it is NP-complete to maximize the number of cleared
rows, minimize the maximum height of an occupied square, or maximize the
number of pieces placed before the game ends. This suggests that the online
version should be computationally difficult.

Second, Tetris represents precisely the kind of large and unstructured MDP for
which it is difficult to design heuristic controllers, and hence policies
designed by ADP algorithms are particularly relevant. Moreover, Tetris has
been employed by a number of researchers as a testbed problem. One of the
important steps in applying these techniques is the choice of basis
functions. Fortunately, there is a \emph{fixed set of basis functions}, to be
described shortly, which have been used by researchers while applying
temporal-difference learning \citep{BertsIoffe,BertsekasNDP}, policy gradient
methods \citep{Kakade}, and approximate linear programming
\citep{FariasVanRoyTetris}. Hence, application of SALP to Tetris allows us to
make a clear comparison to other ADP methods.



The SALP methodology described in
Section~\ref{sec:prac-imp} was applied as follows:

\begin{itemize}
\item {\bf MDP formulation.} We used the formulation of Tetris as a Markov
  decision problem of \citet{FariasVanRoyTetris}. Here, the `state' at a
  particular time encodes the current board configuration and the shape of the
  next falling piece, while the `action' determines the placement of the
  falling piece.

\item {\bf Basis functions.} We employed the 22 basis functions originally
  introduced by \citet{BertsIoffe}. Each basis function takes a Tetris board
  configuration as its argument. The functions are as follows:
\begin{itemize}
\item Ten basis functions, $\phi_0, \hdots, \phi_9$, mapping the state to the
  height $h_k$ of each of the ten columns.
\item Nine basis functions, $\phi_{10},\hdots, \phi_{18}$, each mapping the
  state to the absolute difference between heights of successive columns:
  $|h_{k+1} - h_k|, k=1,\hdots,9$.
\item One basis function, $\phi_{19}$, that maps state to the maximum column
  height: $\max_k h_k$
\item One basis function, $\phi_{20}$, that maps state to the number of 'holes'
  in the board.
\item One basis function, $\phi_{21}$, that is equal to $1$ in every state.
\end{itemize}

\item {\bf State sampling.} Given a sample size $S$, a collection
$\hat{\Xscr}\subset\Xscr$ of $S$ states was sampled. These sampled were
generated in an i.i.d. fashion from the stationary distribution of a (rather
poor) baseline policy\footnote{Our baseline policy had an average performance
  of $113$ points.}.  For each choice of sample size $S$, ten different
collections of $S$ samples were generated.

\item {\bf Optimization.} Given the collection $\hat{\Xscr}$ of sampled
states, an increasing sequence of choices of the violation budget $\theta \geq
0$ is considered. For each choice of $\theta$, the optimization program
\eqref{eq:rlp} was solved.

\item {\bf Policy evaluation.} Given a vector of weights $\hat{r}$, the
performance of the corresponding policy was evaluated using Monte Carlo
simulation. We calculate the average performance of policy $\mu_{\hat{r}}$
over a series of 3000 games. Performance in measured in terms of the average
number of lines eliminated in a single game. The sequence of pieces in each of
the 3000 games was fixed across the evaluation of different policies in order
to allow better comparisons.
\end{itemize}

For each pair $(S,\theta)$, the resulting {\em average} performance (averaged
over each of the 10 policies arising from the different sets of sampled
states) is shown in Figure~\ref{fig:policy-performance}.  Note that the
$\theta=0$ curve in Figure~\ref{fig:policy-performance} corresponds to the
original ALP algorithm. Figure~\ref{fig:policy-performance} provided
experimental evidence for the intuition expressed in Section~\ref{sec:salp}
and the analytic result of Theorem~~\ref{th:bound1}: Relaxing the constraints
of the ALP even slightly, by allowing for a small slack budget, allows for
better policy performance. As the slack budget $\theta$ is increased from $0$,
performance dramatically improves. At the peak value of $\theta=0.16384$, the
SALP generates policies with performance that is an order of magnitude better
than ALP. Beyond this value, the performance of the SALP begins to degrade, as
shown by the $\theta=0.65536$ curve. Hence, we did not explore larger values
of $\theta$.

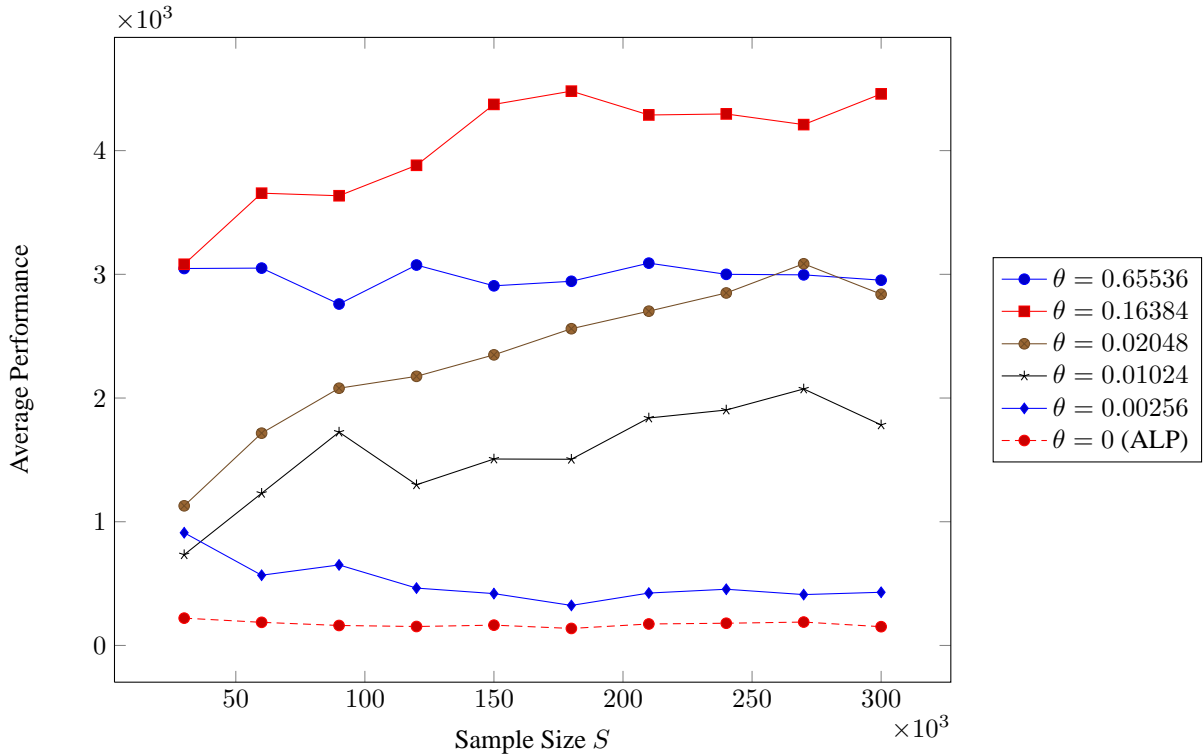
\begin{figure}[htb]
  \centering
  \hidefastcompile{
  \begin{tikzpicture}[font=\small]
    \begin{axis}[xlabel=Sample Size $S$,
      ylabel=Average Performance,
      legend style={anchor=west,at={(1.05,0.5)}},
      xtick scale label code/.code={$\times 10^{#1}$},
      ytick scale label code/.code={$\times 10^{#1}$},
      scaled ticks=base 10:-3,
      width=5in,
      height=4in
      ]

    \addplot coordinates {
        (30000, 3047.001567) 	(60000, 3050.2664) 	(90000, 2759.855867) 	 (120000, 3075.0607)   		(150000, 2907.064833)   		 (180000,2943.8121)     		 (210000,3090.470967)  		 (240000,2999.939033)		(270000,2995.792233)   		 (300000,2952.626967)
      };
      \addlegendentry{$\theta=0.65536$}

      \addplot coordinates {
        (30000, 3081.425133) 	(60000, 3655.9772) 	(90000,3635.192667) 	 (120000,3881.401367)   		(150000,4373.632633)   		 (180000,4480.860167)     		 (210000,4288.057767)  		 (240000,4296.292467)		(270000,4210.255333)   		 (300000,4458.439933)
      };
      \addlegendentry{$\theta=0.16384$}

      \addplot coordinates {
    (30000, 1128.684) 	(60000,1715.9494) 	(90000, 2079.390367) 	 (120000,2175.354133)   		(150000,2349.035267 )   		 (180000,2560.2217)     		 (210000,2702.053767)  		(240000,2849.260333)		 (270000,3085.215167)   		 (300000,2839.2351)
      };
      \addlegendentry{$\theta=0.02048$}

     \addplot coordinates {
    (30000,733.5891333 ) 	(60000, 1229.9427) 	(90000, 1724.4019) 	 (120000,1298.7128)   		(150000,1506.974133 )   		 (180000,1505.043167)     		 (210000,1838.6183)  		 (240000,1903.1347)		(270000,2073.311933)   		 (300000,1782.931367)
      };
      \addlegendentry{$\theta=0.01024$}

      \addplot coordinates {
    (30000, 910.9161333) 	(60000, 567.5833333) 	(90000, 650.8408333) 	 (120000,463.6686667)   		(150000, 419.0497333)   		 (180000,323.4779)     		 (210000,423.4760333)  		(240000,454.3208667)		(270000,411.0994333)   		 (300000,429.8631)
      };
      \addlegendentry{$\theta=0.00256$}

      \addplot coordinates {
        (30000, 220.7153667) (60000, 186.8311) (90000,161.3632667) 	 (120000,153.0669333)   		(150000,164.1075667)   		(180000,137.4048)     		 (210000,173.5442667)  		(240000,179.5007)		(270000,188.9760333)   		 (300000,151.5641667)
      };
      \addlegendentry{$\theta=0$ (ALP)}

    \end{axis}
  \end{tikzpicture}
  }
  \caption{Performance of the average SALP policy 
    for different values of the number of
    sampled states $S$ and the violation budget $\theta$. Values for $\theta$
    were chosen in an increasing fashion starting from $0$,
    until the resulting average performance began to degrade.
    \label{fig:policy-performance}}
\end{figure}

Table~\ref{tab:comparison} summarizes the performance of \emph{best} policies
obtained by various ADP algorithms. Note that all of these algorithms employ
the same basis function architecture. The ALP and SALP results are from our
experiments, while the other results are from the literature. The best
performance result of SALP is a factor of 2 better than the competitors.

\begin{table}[htp]
    \begin{center}
    \begin{tabular}{ccc}
      \toprule
      \textbf{Algorithm} & \textbf{Best Performance} & \textbf{CPU Time} \\
      \midrule
      ALP & 897 & hours \\
      TD-Learning \citep{BertsIoffe} & 3,183 & minutes \\
      ALP with bootstrapping \citep{FariasVanRoyTetris} & 4,274 & hours \\
      TD-Learning \citep{BertsekasNDP} & 4,471 & minutes \\
      Policy gradient \citep{Kakade} & 5,500 & days \\
      SALP & 10,775 & hours \\
      \bottomrule
    \end{tabular}
    \caption{Comparison of the performance of the best policy found with
      various ADP methods.\label{tab:comparison}}
\end{center}
\end{table}

Note that significantly better policies are possible with this basis function
architecture than \emph{any} of the ADP algorithms in
Table~\ref{tab:comparison} discover. Using a heuristic global optimization
method, \citet{Szita2006} report finding policies with a remarkable average
performance of $350{,}000$. Their method is very computationally intensive,
however, requiring one month of CPU time. In addition, the approach employs a
number of rather arbitrary Tetris specific `modifications' that are ultimately
seen to be critical to performance --- in the absence of these modifications,
the method is unable to find a policy for Tetris that scores above a few
hundred points. More generally, global optimization methods typically require
significant trial and error and other problem specific experimentation in
order to work well.

\section{Conclusion}\label{sec:conclusion}

The approximate linear programming (ALP) approach to approximate DP is
interesting at the outset for two reasons. First, the ability to leverage
commercial linear programming software to solve large ADP problems, and
second, the ability to prove rigorous approximation guarantees and performance
bounds. This paper asked whether the formulation considered in the ALP
approach was the ideal formulation. In particular, we asked whether certain
strong restrictions imposed on approximations produced by the approach can be
relaxed in a tractable fashion and whether such a relaxation has a beneficial
impact on the quality of the approximation produced. We have answered both of
these questions in the affirmative. In particular, we have presented a novel
linear programming formulation that, while remaining no less tractable than
the ALP, appears to yield substantial performance gains and permits us to
prove extremely strong approximation and performance guarantees.

There are a number of interesting algorithmic directions that warrant
exploration. For instance, notice that from \eqref{eq:stoch_prg}, that the
SALP may be written as an unconstrained stochastic optimization problem. Such
problems suggest natural \emph{online} update rules for the weights $r$, based
on stochastic gradient methods, yielding `data-driven' ADP methods. The
menagerie of online ADP algorithms available at present are effectively
iterative methods for solving a projected version of Bellman's
equation. TD-learning is a good representative of this type of approach and,
as can be seen from Table~\ref{tab:comparison}, is not among the highest
performing algorithms in our computational study. An online update rule that
effectively solves the SALP promises policies that will perform on par with
the SALP solution, while at the same time retaining the benefits of an online
ADP algorithm. A second interesting algorithmic direction worth exploring is
an extension of the smoothed linear programming approach to average cost
dynamic programming problems.

As discussed in Section~\ref{sec:analysis}, theoretical guarantees for ADP
algorithms typically rely on some sort of idealized assumption. For instance,
in the case of the ALP, it is the ability to solve an LP with a potentially
intractable number of states or else access to a set of sampled states,
sampled according to some idealized sampling distribution. For the SALP, it is
the latter of the two assumptions. It would be interesting to see whether this
assumption can be loosened for some specific class of MDPs. An interesting
class of MDPs in this vein are high dimensional optimal stopping problems. Yet
another direction for research, is understanding the dynamics of
`bootstrapping' procedures, that solve a sequence of sampled versions of the
SALP with samples for a given SALP in the sequence drawn according to a
policy produced by the previous SALP is the sequence.

{\small
\singlespacing
\bibliography{ConcatBib}
}

\appendix

\section{Proofs for Section~\ref{sec:simple_bound}}\label{sec:ell-proof}

\begin{le:ell} 
  For any $r \in \R^K$ and $\theta \geq 0$:
  \begin{enumerate}[(i)]
  \item $\ell(r,\theta)$ is a finite-valued, decreasing, piecewise linear,
    convex function of $\theta$.
  \item
    \[
    \ell(r,\theta) \leq \frac{1+\alpha}{1-\alpha} \|  J^* - \Phi r\|_\infty.
    \]
  \item
    The right partial derivative of $\ell(r,\theta)$ with
    respect to $\theta$ satisfies
    \[
    \frac{\partial^+}{\partial \theta^+} \ell(r,0) 
    = - \left((1 - \alpha) \sum_{x\in\Omega(r)} \pi_{\mu^*,\nu}(x)\right)^{-1},
    \]
    where 
    \[
    \Omega(r) \defeq \argmax_{x\in\Xscr}\ \Phi r(x) - T\Phi r (x).
    \]
  \end{enumerate}
\end{le:ell}
\begin{proof}
  (i) Given any $r$, clearly $\gamma \defeq \|\Phi r - T\Phi r\|_\infty$,
  $s\defeq 0$ is a feasible point for \eqref{eq:flp}, so $\ell(r,\theta)$ is
  well-defined. To see that the LP is bounded, suppose $(s,\gamma)$ is
  feasible. Then, for any $x\in\Xscr$ with $\pi_{\mu^*,\nu}(x)>0$,
  \[
  \gamma \geq \Phi r(x) - T\Phi r(x) - s(x) \geq
  \Phi r(x) - T\Phi r(x) - \theta/\pi_{\mu^*,\nu}(x) > \infty.
  \]
  Letting $(\gamma_1, s_1)$ and $(\gamma_2, s_2)$ represent optimal solutions
  for the LP \eqref{eq:flp} with parameters $(r,\theta_1)$ and $(r,\theta_2)$
  respectively, it is easy to see that $((\gamma_1 + \gamma_2)/2,
  (s_1+s_2)/2)$ is feasible for the LP with parameters
  $(r,(\theta_1+\theta_2)/2)$. It follows that $\ell(r,(\theta_1 +
  \theta_2)/2) \leq (\ell(r,\theta_1) + \ell(r,\theta_2))/2$. The remaining
  properties are simple to check.
  
  (ii) Let $\epsilon\defeq \|  J^* - \Phi r\|_\infty$. Then,
  \[
  \|  T \Phi r - \Phi r  \|_\infty
  \leq \|  J^* - T \Phi r \|_\infty + \| J^* - \Phi r \|_\infty
  \leq \alpha \|  J^*  - \Phi r \|_\infty + \epsilon = (1+\alpha)\epsilon.
  \]
  Since $\gamma \defeq \| T \Phi r - \Phi r \|_\infty$, $s \defeq 0$ is
  feasible for \eqref{eq:flp}, the result follows.
  
  (iii) Fix $r\in\R^K$, and define
  \[
  \Delta \defeq \max_{x \in \Xscr}\ \big(\Phi r (x) - T \Phi r (x)\big) 
  - \max_{x \in \Xscr \setminus \Omega(r)}\ \big(\Phi r (x) - T \Phi r (x)\big) > 0.
  \]
  Consider the program for $\ell(r, \delta)$. It is easy to verify that for
  $\delta \geq 0$ and sufficiently small, viz. $\delta \leq \Delta \sum_{x \in
    \Omega(r)} \pi_{\mu^*,\nu}(x)$, $(\bar{s}_{\delta}, \bar{\gamma}_\delta)$
  is an optimal solution to the program, where
  \[
  \bar{s}_\delta(x) \defeq 
  \begin{cases}
  \frac{\delta}{\sum_{x \in \Omega(r)} \pi_{\mu^*,\nu}(x)}
  & \text{if $x \in \Omega(r)$,} \\
  0 & \text{otherwise,} 
  \end{cases}
  \]
  and
  \[
  \bar{\gamma}_\delta \defeq \gamma_0 - \frac{\delta}{\sum_{x \in \Omega(r)}
  \pi_{\mu^*,\nu}(x)},
  \]
  so that 
  \[
  \ell(r,\delta) = \ell(r,0) - \frac{\delta}{(1-\alpha)\sum_{x \in \Omega(r)} \pi_{\mu^*,\nu}(x)}.
  \]
  Thus, 
  \[
  \frac{\ell(r,\delta) - \ell(r,0)}{\delta} = -\left((1-\alpha)\sum_{x \in \Omega(r)} \pi_{\mu^*,\nu}(x)\right)^{-1}.
  \]
  Taking a limit as $\delta \searrow 0$ yields the result. 
\end{proof}

\begin{le:feas}
  Suppose that the vectors $J\in\R^\Xscr$ and $s\in\R^\Xscr$ satisfy
  \[
  J \leq T_{\mu^*} J + s.
  \]
  Then,
  \[
  J  \leq J^* +  \Delta^* s,
  \]
  where
  \[
  \Delta^* \defeq
  \sum_{k=0}^\infty (\alpha P_{\mu^*})^k = (I - \alpha P_{\mu^*})^{-1},
  \]
  and $P_{\mu^*}$ is the transition probability matrix corresponding to an
  optimal policy.

  In particular, if $(r,s)$ is feasible for the LP
  \eqref{eq:salp3}. Then,
  \[
  \Phi r  \leq J^* +  \Delta^* s.
  \]
\end{le:feas}
\begin{proof}
  Note that the $T_{\mu^*}$, the Bellman operator corresponding to the optimal
  policy $\mu^*$, is monotonic and is a contraction. Then, repeatedly applying
  $T_{\mu^*}$ to the inequality $J \leq T_{\mu^*} J + s$ and using the fact
  that $T^k_{\mu^*} J \rightarrow J^*$, we obtain
  \[
  J \leq J^* + \sum_{k=0}^\infty (\alpha
  P_{\mu^*})^k s = J^* + \Delta^* s.
  \]
\end{proof}

\section{Proof of Lemma~\ref{le:sample_complexity}}
\label{sec:sampling-proof}




We begin with the following definition: consider a family $\Fscr$ of functions
from a set $\Sscr$ to $\{0,1\}$. Define the \emph{Vapnik-Chervonenkis (VC)
  dimension} $\dim_{\text{VC}}(\Fscr)$ to be the cardinality $d$ of the
largest set $\{x_1,x_2,\dots,x_d \} \subset \Sscr$ satisfying:
\[
\forall e \in \{0,1\}^d,\ 
\exists f \in \Fscr\text{ such that }
\forall i,\ f(x_i) = 1\text{ iff } e_i=1.
\]

Now, let $\Fscr$ be some set of \emph{real}-valued functions mapping $\Sscr$
to $[0,B]$. The \emph{pseudo-dimension} $\dim_P(\Fscr)$ is the following
generalization of VC dimension: for each function $f \in \Fscr$ and scalar $c
\in \R$, define a function $g\colon \Sscr \times \R \tends \{0,1\}$ according
to:
\[
g(x,c) \defeq \I{f(x) - c \geq 0}.
\]
Let $\Gscr$ denote the set of all such functions. Then, we define
$\dim_P(\Fscr) \defeq \dim_{\text{VC}}(\Gscr)$.

In order to prove Lemma~\ref{le:sample_complexity}, define the $\Fscr$ to be
the set of functions $f\colon \R^K\times\R\rightarrow [0,B]$, where, for all
$x\in\R^K$ and $y\in\R$,
\[
f(y,z) \defeq \zeta\left(r^\top y + z\right).
\]
Here, $\zeta(t) \defeq \max\left(\min(t,B),0\right)$, and $r\in\R^K$ is a
vector that parameterizes $f$. We will show that $\dim_P(\Fscr) \leq K+2$.

We will use the following standard result from convex geometry:
\begin{lemma}[Radon's Lemma] A set $A\subset\R^m$ of $m+2$ points can be
  partitioned into two disjoint sets $A_1$ and $A_2$, such that the convex
  hulls of $A_1$ and $A_2$ intersect.
\end{lemma}

\begin{lemma}
$\dim_P(\Fscr) \leq K+2$
\end{lemma}
\begin{proof}
  Assume, for the sake of contradiction, that $\dim_P(\Fscr) > K+2$. It must
  be that there exists a `shattered' set
  \[
  \left\{
    \big(y^{(1)},z^{(1)},c^{(1)}\big),
    \big(y^{(2)},z^{(2)},c^{(2)}\big),
    \ldots,
    \big(y^{(K+3)},z^{(K+3)},c^{(K+3)}\big)
  \right\}
  \subset \R^K\times\R\times\R,
  \]
  such that, for all $e \in \{0,1\}^{K+3}$, there exists a vector
  $r_e \in \R^K$ with
  \[
  \zeta\left(r_e^\top y^{(i)} + z^{(i)} \right) \geq c^{(i)}
  \text{ iff }
  e_i =1,
  \quad\forall\ 1 \leq i \leq K+3.
  \]

  Observe that we must have $c^{(i)} \in (0,B]$ for all $i$, since if $c^{(i)}
  \leq 0$ or $c^{(i)} > B$, then no such shattered set can be
  demonstrated. But if $c^{(i)} \in (0,B]$, for all $r\in\R^K$,
  \[
  \zeta\left(r^\top y^{(i)} + z^{(i)} \right) \geq c^{(i)}
  \implies
  r_e^\top y^{(i)} \geq c^{(i)} - z^{(i)},
  \]
  and
  \[
  \zeta\left(r^\top y^{(i)} + z^{(i)} \right) < c^{(i)}
  \implies
  r_e^\top y^{(i)} < c^{(i)} - z^{(i)}.
  \]

  For each $1 \leq i \leq K+3$, define $x^{(i)} \in \R^{K+1}$ component-wise
  according to
  \[
  x^{(i)}_j \defeq
  \begin{cases}
    y^{(i)}_j & \text{if $j < K+1$}, \\
    c^{(i)} - z^{(i)} & \text{if $j=K+1$.}
  \end{cases}
  \]
  Let $A = \{x^{(1)},x^{(2)},\ldots,x^{(K+3)}\}\subset\R^{K+1}$, and let $A_1$
  and $A_2$ be subsets of $A$ satisfying the conditions of Radon's
  lemma. Define a vector $\tilde{e} \in \{0,1\}^{K+3}$ component-wise
  according to
  \[
  \tilde{e}_i \defeq \I{x^{(i)}\in A_1}.
  \]
  Define the vector $\tilde{r} \defeq r_{\tilde{e}}$. Then, we have
  \[
  \sum_{j=1}^K \tilde{r}_j x_j \geq x_{K+1},
  \quad\forall\ x\in A_1,
  \]
  \[
  \sum_{j=1}^K \tilde{r}_j x_j < x_{K+1},
  \quad\forall\ x\in A_2.
  \]

  Now, let $\bar{x}\in\R^{K+1}$ be a point contained in both the convex hull of
  $A_1$ and the convex hull of $A_2$. Such a point must exist by Radon's
  lemma. By virtue of being contained in the convex hull of $A_1$, we must
  have
  \[
  \sum_{j=1}^{K} \tilde{r}_j \bar{x}_j \geq \bar{x}_{K+1}.
  \]
  Yet, by virtue of being contained in the convex hull of $A_2$, we must have
  \[
  \sum_{j=1}^{K} \tilde{r}_j \bar{x}_j  < \bar{x}_{K+1},
  \]
  which is impossible.
\end{proof}

With the above pseudo-dimension estimate, 
Lemma~\ref{le:sample_complexity} follows immediately from Corollary~2 of
of \citet[][Section~4]{Haussler92}.

\end{document}